%% file: main.tex
\documentclass{article}
\usepackage{amsmath, amsfonts, amsthm, amssymb}
\usepackage{enumitem}
\usepackage{graphicx} 
\usepackage{geometry}
\usepackage[english]{babel}
\usepackage[affil-it]{authblk}
\usepackage[T1]{fontenc}
\usepackage[utf8]{inputenc}
\usepackage{lmodern}
\usepackage{tikz}
\usetikzlibrary{arrows.meta,positioning}
\usepackage{booktabs}
\usepackage{tabularx}
\usepackage{caption}
\usepackage{import}
\usepackage{bbm}
\usepackage{hyperref}
\title{Almost Disjointness Principles and $Q$-Space Cardinals}
\date{}
\theoremstyle{definition}
\newtheorem{definition}{Definition}
\newtheorem{prob}[definition]{Problem}
\newtheorem{thm}[definition]{Theorem}
\newtheorem{claim}{Claim}[definition]
\newtheorem{lemma}[definition]{Lemma}
\newtheorem{proposition}[definition]{Proposition}

\DeclareMathOperator{\dom}{dom}
\DeclareMathOperator{\Col}{Col}
\DeclareMathOperator{\htt}{ht}
\author{Vinicius de Oliveira Rodrigues\thanks{Electronic address: \texttt{vinior@ime.usp.br}}}
\affil{University of São Paulo}

\begin{document}\maketitle

\begin{abstract}
    Banakh and Bazylevych introduced separation-axiom variants $\mathfrak q_i$, for $i=1,2,2\frac{1}{2}$, of the cardinal $\mathfrak q$, together with a cardinal $\mathfrak{adp}$ lying between $\mathfrak{dp}$ and $\mathfrak{ap}$.
    They asked whether $\mathfrak{adp}$ coincides with either of these two cardinals.
    We prove in ZFC that $\mathfrak{adp}=\mathfrak{dp}$.
    We define a dual variant $\mathfrak{adp}_2$ and show that $\mathfrak{adp}_2=\mathfrak{ap}$.

    We further study the relation between $\mathfrak{ap}$ and the weakened $Q$-space cardinals.
    We introduce a tree analogue $\mathfrak{at}$ of $\mathfrak{ap}$ and prove $\mathfrak{at}=\mathfrak q_{2}$, concluding that $\mathfrak{ap}\leq\mathfrak q_{2}$, which answers a question of Banakh and Bazylevych.

    Assuming the Generalized Continuum Hypothesis, we construct ccc forcing extensions with $\mathfrak{ap}=\omega_1<\mathfrak{at}=\mathfrak c$, so $\mathfrak{ap}<\mathfrak q_2$ is consistent with ZFC.

\end{abstract}

\subimport{sections/}{introduction.tex}
\subimport{sections/}{adp.tex}
\subimport{sections/}{basics.tex}
\subimport{sections/}{diagonal.tex}
\subimport{sections/}{forcing.tex}
\subimport{sections/}{conclusion.tex}
\bibliographystyle{plain}
\bibliography{includes/biblio.bib}
\end{document}

%% file: sections/introduction.tex
\section{Introduction}

Cardinal characteristics of the continuum are cardinals between $\aleph_1$,
the first uncountable cardinal, and $\mathfrak c$, the cardinality of the
continuum. They measure combinatorial properties of the real numbers; see
\cite{blass2009combinatorial} for background.

The symbol $\omega$ denotes the set of natural numbers, which includes $0$ and is identified with the first infinite ordinal.

The cardinal characteristics $\mathfrak{ap}$ and $\mathfrak{dp}$, both related
to the pseudointersection number $\mathfrak p$, were studied by Brendle in
\cite{brendle1999dow}. We recall the relevant definitions.

\begin{itemize}
    \item For a set $S$, $[S]^{<\omega}$ denotes the collection of all finite subsets of $S$, and $[S]^\omega$ denotes the collection of all countably infinite subsets of $S$.
    \item Sets $a,b\in[\omega]^{\aleph_0}$ are orthogonal, written $a\perp b$, if $a\cap b$ is finite.
    \item Collections $\mathcal A,\mathcal B\subseteq[\omega]^{\aleph_0}$ are orthogonal, written $\mathcal A\perp\mathcal B$, if $a\perp b$ for every $a\in\mathcal A$ and $b\in\mathcal B$.
    \item If $\mathcal A,\mathcal B\subseteq[\omega]^{\aleph_0}$, then $\mathcal A$ is weakly separated from $\mathcal B$ if there exists $X\subseteq\omega$ such that $a\not\perp X$ for every $a\in\mathcal A$, and $b\perp X$ for every $b\in\mathcal B$. In this case, we say that $X$ weakly separates $(\mathcal A,\mathcal B)$.
    \item $\mathcal A\subseteq[\omega]^{\aleph_0}$ is an almost disjoint family if $a\perp b$ for every two distinct $a,b\in\mathcal A$.
\end{itemize}

Our convention regarding the order of weak separation follows Banakh and Bazylevych \cite{banakh2023q}.  It is the reverse of Brendle's convention for ordered pairs \cite{brendle1999dow}.

Since weak separation is not symmetric, the order of the two families matters in
the following definitions.

\begin{itemize}
    \item $\mathfrak{dp}$ is the least cardinal $\kappa$ for which there exist $\mathcal A,\mathcal B\subseteq[\omega]^{\aleph_0}$ such that $\mathcal A\perp\mathcal B$, $|\mathcal A\cup\mathcal B|=\kappa$, and $\mathcal A$ is not weakly separated from $\mathcal B$.
    \item $\mathfrak{ap}$ is the least cardinal $\kappa$ for which there exist $\mathcal A,\mathcal B\subseteq[\omega]^{\aleph_0}$ such that $\mathcal A\perp\mathcal B$, $|\mathcal A\cup\mathcal B|=\kappa$, $\mathcal A\cup\mathcal B$ is an almost disjoint family, and $\mathcal A$ is not weakly separated from $\mathcal B$.
\end{itemize}

It is immediate that $\mathfrak{dp}\leq\mathfrak{ap}$.
This inequality can be strict: Brendle showed by forcing that
$\mathfrak{dp}<\mathfrak{ap}$ is consistent \cite{brendle1999dow}.
Both cardinals are greater than or equal to $\mathfrak p$, and consistently
strictly greater than $\mathfrak p$ \cite{brendle1999dow}. Here $\mathfrak p$ is the least
cardinality of a centered family of infinite subsets of $\omega$ with no
pseudointersection; that is, every nonempty finite subfamily has infinite
intersection, but there is no single infinite set $P$ such that
$P\setminus A$ is finite for every member $A$ of the family.

If $X$ is a topological space and $Y\subseteq X$, we say that $Y$ is a $G_\delta$ subset of $X$ if there are open subsets $(U_n)_{n \in \omega}$ of $X$ such that $Y=\bigcap_{n<\omega} U_n$.
Dually, $Y$ is an $F_\sigma$ subset of $X$ if it is a countable union of closed subsets of $X$.
We say that $X$ has a $G_\delta$-diagonal if the set $\Delta_X=\{(x,x):x\in X\}$ is a $G_\delta$ as a subset of $X^2$.

A $Q$-set is classically defined as an uncountable set of reals all of whose subsets are
$G_\delta$ sets in the subspace topology (equivalently, all of its subsets are
$F_\sigma$ sets in the subspace topology) \cite{miller1984special,fleissner1980qsets}. The existence of $Q$-sets is independent of the Zermelo--Fraenkel set theory with the axiom of choice (ZFC) \cite{fleissner1980qsets,miller1984special}. Indeed, if
there is a $Q$-set, then it has at least $2^{\aleph_1}$ subsets, all realized as
traces of $G_\delta$ subsets of the reals. Since there are only $\mathfrak c$
many $G_\delta$ subsets of the reals, no $Q$-set can exist under the Continuum
Hypothesis (CH). On the other hand, every set of reals of cardinality less than
$\mathfrak{ap}$ has all of its subsets $G_\delta$ in the subspace topology
\cite{brendle1999dow}; hence, if
$\aleph_1<\mathfrak{ap}=\mathfrak c$, as happens for example under Martin's
Axiom together with the negation of CH \cite{blass2009combinatorial}, then $Q$-sets
exist.

$Q$-sets are also related to the separable case of the Normal Moore Space
Conjecture: a non-metrizable separable normal Moore space exists if and only if
there is a $Q$-set \cite{heath1964screenability}.

The smallest cardinality of a non-$Q$-set is denoted by $\mathfrak q$
\cite{miller1984special,brendle1999dow}.
More generally, a $Q$-space is a topological space in which every subset is a
$G_\delta$ set. In this terminology, $\mathfrak q$ is the least cardinality of
a metrizable separable non-$Q$-space, equivalently, of a regular second-countable
non-$Q$-space.

These notions have been studied in more general settings. There are two natural
directions to pursue: weaken regularity to weaker separation axioms, or weaken
second countability. In \cite{banakh2023q}, Banakh and Bazylevych followed the
first direction and defined the cardinals $\mathfrak q_1$, $\mathfrak q_2$ and
$\mathfrak q_{2\frac{1}{2}}$ by considering second-countable non-$Q$-spaces
which are, respectively, $T_1$, Hausdorff and Urysohn. Recall that a topological
space is $T_1$ if every singleton is closed; Hausdorff, or $T_2$, if every two
distinct points have disjoint neighborhoods; and Urysohn, or
$T_{2\frac{1}{2}}$, if every two distinct points have disjoint closed
neighborhoods.

Banakh and Bazylevych also introduced a cardinal $\mathfrak{adp}$, which lies
between $\mathfrak{dp}$ and $\mathfrak{ap}$ and is useful in the study of
$Q$-spaces.

\begin{definition}
    $\mathfrak{adp}$ is the least cardinal $\kappa$ for which there exist $\mathcal A,\mathcal B\subseteq[\omega]^{\aleph_0}$ such that $\mathcal A\perp\mathcal B$, $|\mathcal A\cup\mathcal B|=\kappa$, $\mathcal A$ is an almost disjoint family, and $\mathcal A$ is not weakly separated from $\mathcal B$.
\end{definition}

They asked the following.

\begin{prob}[{\cite[Problem 12]{banakh2023q}}] Is $\mathfrak{adp}=\mathfrak{dp}$? Is $\mathfrak{adp}=\mathfrak{ap}$?
\end{prob}

In Section~\ref{sec:adp}, we answer the first part by proving
$\mathfrak{adp}=\mathfrak{dp}$ in ZFC. Together with the consistency of
$\mathfrak{dp}<\mathfrak{ap}$, this shows that
$\mathfrak{adp}=\mathfrak{ap}$ is not provable in ZFC.

The asymmetry of weak separation also suggests asking what happens if the
almost disjointness requirement in the definition of $\mathfrak{adp}$ is imposed
on $\mathcal B$ instead of on $\mathcal A$.

\begin{definition}
    $\mathfrak{adp}_2$ is the least cardinal $\kappa$ for which there exist $\mathcal A,\mathcal B\subseteq[\omega]^{\aleph_0}$ such that $\mathcal A\perp\mathcal B$, $|\mathcal A\cup\mathcal B|=\kappa$, $\mathcal B$ is an almost disjoint family, and $\mathcal A$ is not weakly separated from $\mathcal B$.
\end{definition}

Using the same technique, we show that $\mathfrak{adp}_2=\mathfrak{ap}$.
Thus, in every model in which $\mathfrak{dp}<\mathfrak{ap}$, this modified
cardinal is distinct from $\mathfrak{dp}$.

Banakh and Bazylevych also asked:

\begin{prob}[{\cite[Problem 11(1)]{banakh2023q}}] Is $\mathfrak{ap}\leq\mathfrak q_2$?
\end{prob}

In Section~\ref{sec:hausdorff-q-spaces}, we provide a positive answer.
We introduce a tree analogue of
$\mathfrak{ap}$, denoted by $\mathfrak{at}$. This cardinal is obtained by
replacing almost disjoint families of subsets of $\omega$ with almost
disjoint families of infinite subtrees of an $\omega$-tree. It satisfies
$\mathfrak{ap}\leq\mathfrak{at}$, and we show that $\mathfrak{at}=\mathfrak q_2$ (Theorem~\ref{thm:at-equals-q2}), yielding $\mathfrak{ap}\leq\mathfrak q_2$.

In Section~\ref{sec:small-diagonals}, we define the variants
$\mathfrak q_{i\Delta}$ obtained by additionally requiring a
$G_\delta$-diagonal. We prove
$\mathfrak q_{1\Delta}=\mathfrak{at}=\mathfrak q_2$ and
$\mathfrak q_{2\frac{1}{2}\Delta}=\mathfrak q_{2\frac{1}{2}}$.

Finally, in Section~\ref{sec:ap-below-at-model}, we show that the tree cardinal
$\mathfrak{at}$ is not simply another presentation of $\mathfrak{ap}$.
Assuming the Generalized Continuum Hypothesis (GCH), for every regular cardinal $\lambda>\omega_1$ we construct a ccc
forcing extension in which
\[
    \mathfrak{ap}=\omega_1<
    \mathfrak{at}=\mathfrak q_{2}=\mathfrak c=\lambda.
\]
In particular, $\mathfrak{ap}<\mathfrak q_2$ is consistent.
The construction is based on Brendle's model for $\mathfrak{ap}<\mathfrak q$.

\subsection*{Summary of results}

We end the introduction with a summary of the main results of this paper.

\begin{itemize}
    \item We prove $\mathfrak{adp}=\mathfrak{dp}$ in ZFC (Theorem~\ref{thm:adp-equals-dp}), answering the first question in \cite[Problem~12]{banakh2023q}. We define a dual variant, $\mathfrak{adp}_2$, and we prove $\mathfrak{adp}_2=\mathfrak{ap}$ (Theorem~\ref{thm:adp2}).

    \item We introduce the tree cardinal $\mathfrak{at}$ and prove $\mathfrak{ap}\leq\mathfrak{at}=\mathfrak q_2$ (Theorem~\ref{thm:at-equals-q2}). This gives a combinatorial characterization of $\mathfrak q_2$ and answers \cite[Problem~11(1)]{banakh2023q} positively: $\mathfrak{ap}\leq\mathfrak q_2$ holds in ZFC.

    \item We determine the effect of requiring a $G_\delta$-diagonal in the definition of $Q$-space cardinals. We prove $\mathfrak q_{1\Delta}=\mathfrak{at}=\mathfrak q_2$ (Theorem~\ref{thm:q1delta-equals-at}). We also show that every second-countable Urysohn space has a $G_\delta$-diagonal, yielding $\mathfrak q_{2\frac{1}{2}\Delta}=\mathfrak q_{2\frac{1}{2}}$ (Lemma~\ref{lem:second-countable-urysohn-gdelta-diagonal}).

    \item We show that the inequality $\mathfrak{ap}\leq\mathfrak q_2$ is consistently strict. Assuming GCH, for every regular $\lambda>\omega_1$ we construct a ccc forcing extension with $\mathfrak{ap}=\omega_1<\mathfrak{at}=\mathfrak q_2=\mathfrak q_{2\frac{1}{2}}=\mathfrak c=\lambda$ (Theorem~\ref{thm:consistent-ap-below-at}). Thus the inequality answering \cite[Problem~11(1)]{banakh2023q} cannot be strengthened to an equality in ZFC.
\end{itemize}

%% file: sections/adp.tex
\section{On $\mathfrak{adp}$ and $\mathfrak{adp}_2$}\label{sec:adp}

We will need the following result, which is a particular case of \cite[Theorem 2.1]{zbMATH03494419}.
For completeness, we sketch a direct proof for this specific case. Recall that almost disjoint families of cardinality $\mathfrak c$ exist: for instance, for each $x \in 2^\omega$, consider $a_x=\{x\restriction n: n \in \omega\}$.
Then $\{a_x: x \in 2^{\omega}\}$ is an almost disjoint family over the infinite countable set of all finite binary sequences.

\begin{thm}\label{maintool}
    Let $\lambda<\mathfrak c$ and $(a_\alpha: \alpha<\lambda)$ be a family in $[\omega]^{\aleph_0}$.
    Then there exists a family $(a'_\alpha: \alpha<\lambda)$ such that $a'_\alpha\in[a_\alpha]^{\aleph_0}$ for every $\alpha<\lambda$, and $a'_\alpha\perp a'_\beta$ whenever $\alpha\neq\beta$.
\end{thm}

\begin{proof}
    We may assume that $\lambda$ is infinite.
    
    For each $\alpha<\lambda$, let $\mathcal A_\alpha=(b^\alpha_\xi: \xi<\lambda^+)$ be an injective indexing of an almost disjoint family of cardinality $\lambda^+$ over the infinite countable set $a_\alpha$.

    Let $F:\lambda \rightarrow \lambda$ be such that $F(\alpha)=\min\{\beta\leq \alpha: |\{\xi<\lambda^+: b^\beta_\xi\not\perp a_\alpha\}|=\lambda^+\}$. Let $Y=\bigcup_{\alpha, \beta<\lambda}\{\xi<\lambda^+: b_\xi^\beta\not\perp a_\alpha\,\wedge\, F(\alpha)>\beta\}$. By the minimality of $F(\alpha)$, the regularity of $\lambda^+$, and the fact that $\lambda$ is infinite, $|Y|\leq \lambda$.
    
    For each $\alpha<\lambda$, let $X_\alpha=F^{-1}[\{\alpha\}]$. Recursively define, for $\alpha<\lambda$, an injective $G_\alpha: X_\alpha\rightarrow \lambda^+$ and $(a_\delta': \delta\in X_\alpha)$ so that $G_\alpha(\delta)\in \{\xi<\lambda^+: b_\xi^\alpha\not \perp a_\delta\}\setminus Y$ and $a_\delta'=a_\delta\cap b^{\alpha}_{G_\alpha(\delta)}$.
    This is possible because, for each $\delta\in X_\alpha$, the set $\{\xi<\lambda^+: b_\xi^\alpha\not \perp a_\delta\}$ has size $\lambda^+$, while at each stage at most $\lambda$ indices have been excluded.
    If $\delta,\gamma\in X_\alpha$ are distinct, then the injectivity of $G_\alpha$ and the almost disjointness of $\mathcal A_\alpha$ give $a'_\delta\perp a'_\gamma$.
    If $\delta\in X_\alpha$, $\gamma\in X_\beta$, and $\alpha<\beta$, then $G_\alpha(\delta)\notin Y$ implies $b^\alpha_{G_\alpha(\delta)}\perp a_\gamma$, and hence $a'_\delta\perp a'_\gamma$; the case $\beta<\alpha$ is symmetric.
\end{proof}

Now we are ready to prove the two equalities announced in the introduction.
First, we settle Problem 12 of \cite{banakh2023q}.

\begin{thm}\label{thm:adp-equals-dp}
     In ZFC,
    $\mathfrak{dp}=\mathfrak{adp}$.
\end{thm}
\begin{proof}
    It is clear that $\mathfrak{dp}\leq\mathfrak{adp}$. For the converse, assume $\kappa<\mathfrak{adp}$.
    We show that $\kappa<\mathfrak{dp}$.

    Let $\mathcal A, \mathcal B\subseteq[\omega]^{\aleph_0}$ be such that $\mathcal A\perp\mathcal B$ and $|\mathcal A\cup\mathcal B|\leq\kappa$.
    We show that $\mathcal A$ is weakly separated from $\mathcal B$.
    Let $\lambda=|\mathcal A|\leq \kappa$ and enumerate it as $\mathcal A=\{a_\alpha: \alpha<\lambda\}$.

    Since $\mathfrak{adp}\leq\mathfrak{ap}\leq\mathfrak c$, we have
    $\lambda\leq\kappa<\mathfrak{adp}\leq\mathfrak c$.
    Thus, by Theorem \ref{maintool}, there exists $(a'_\alpha: \alpha<\lambda)$ such that $a'_\alpha\in[a_\alpha]^{\aleph_0}$ and for every two distinct
    $\alpha, \beta<\lambda$, $a'_\alpha\perp a'_\beta$. Let $\mathcal A'=\{a'_\alpha: \alpha<\lambda\}$. 
    Then $\mathcal A'\perp \mathcal B$, $\mathcal A'$ is an almost disjoint family and
    $|\mathcal A'\cup\mathcal B|\leq \kappa<\mathfrak{adp}$, so $\mathcal A'$ is weakly separated from $\mathcal B$.

    Let $X\subseteq \omega$ be such that $a'_\alpha\not\perp X$ for every $\alpha<\lambda$, and $b\perp X$ for every $b\in\mathcal B$.
    Since $a'_\alpha\subseteq a_\alpha$, we have $a'_\alpha\cap X\subseteq a_\alpha\cap X$ for every $\alpha<\lambda$.
    Thus, $a_\alpha\not\perp X$ for every $\alpha<\lambda$.
    Therefore, $X$ weakly separates $\mathcal A$ from $\mathcal B$.
\end{proof}

As discussed in the introduction, it is natural to ask whether the asymmetric version of $\mathfrak{adp}$, which we named $\mathfrak{adp}_2$, is also equal to $\mathfrak{dp}$.
We show that it is equal to the well-studied cardinal $\mathfrak{ap}$.
Consequently, in every model in which $\mathfrak{dp}<\mathfrak{ap}$, this variant is distinct from $\mathfrak{dp}$.
\begin{thm}\label{thm:adp2}
    In ZFC, $\mathfrak{adp}_2=\mathfrak{ap}$.
\end{thm}
\begin{proof}
    It is clear that $\mathfrak{adp}_2\leq\mathfrak{ap}$. For the converse, assume $\kappa<\mathfrak{ap}$.
    We show that $\kappa<\mathfrak{adp}_2$.

    Let $\mathcal A, \mathcal B\subseteq[\omega]^{\aleph_0}$ be such that $\mathcal A\perp\mathcal B$, $\mathcal B$ is an almost disjoint family and $|\mathcal A\cup\mathcal B|\leq\kappa$.
    We show that $\mathcal A$ is weakly separated from $\mathcal B$.
    Let $\lambda=|\mathcal A|\leq \kappa$ and enumerate it as $\mathcal A=\{a_\alpha: \alpha<\lambda\}$.

    Since $\mathfrak{ap}\leq\mathfrak c$, we have
    $\lambda\leq\kappa<\mathfrak{ap}\leq\mathfrak c$.
    Thus, by Theorem \ref{maintool}, there exists $(a'_\alpha: \alpha<\lambda)$ such that $a'_\alpha\in[a_\alpha]^{\aleph_0}$ and for every two distinct
    $\alpha, \beta<\lambda$, $a'_\alpha\perp a'_\beta$. Let $\mathcal A'=\{a'_\alpha: \alpha<\lambda\}$. 
    Then $\mathcal A'\perp \mathcal B$, $\mathcal A'\cup\mathcal B$ is an almost disjoint family and
    $|\mathcal A'\cup\mathcal B|\leq \kappa<\mathfrak{ap}$, so $\mathcal A'$ is weakly separated from $\mathcal B$.

    Let $X\subseteq \omega$ be such that $a'_\alpha\not\perp X$ for every $\alpha<\lambda$, and $b\perp X$ for every $b\in\mathcal B$.
    Since $a'_\alpha\subseteq a_\alpha$, we have $a'_\alpha\cap X\subseteq a_\alpha\cap X$ for every $\alpha<\lambda$.
    Thus, $a_\alpha\not\perp X$ for every $\alpha<\lambda$.
    Therefore, $X$ weakly separates $\mathcal A$ from $\mathcal B$.
\end{proof}

%% file: sections/basics.tex
\section{On almost disjointness and $Q$-spaces}\label{sec:hausdorff-q-spaces}
As discussed in the introduction, the cardinal $\mathfrak q$ is the least cardinality of a set of reals that is not a $Q$-space. For $f,g\in\omega^\omega$, write $f\leq^* g$ if $f(n)\leq g(n)$ for all but finitely many $n\in\omega$. It is known that $\mathfrak p\leq \mathfrak q\leq \mathfrak b$, where $\mathfrak b$ is the bounding number, namely the least cardinality of a family in $\omega^\omega$ that is not bounded by any single function with respect to $\leq^*$ \cite{miller1984special}.

In \cite{banakh2023q}, Banakh and Bazylevych introduced the following separation-axiom variants of $\mathfrak q$ (with the classical invariant denoted by $\mathfrak q_0$ in their notation).

\begin{definition}\label{def:q-variants}
    The cardinal $\mathfrak q_1$ is the least cardinal $\kappa$ for which there exists a second-countable, $T_1$, non-$Q$-space of cardinality $\kappa$.

    The cardinal $\mathfrak q_2$ is the least cardinal $\kappa$ for which there exists a second-countable, Hausdorff, non-$Q$-space of cardinality $\kappa$.

    The cardinal $\mathfrak q_{2\frac{1}{2}}$ is the least cardinal $\kappa$ for which there exists a second-countable, Urysohn, non-$Q$-space of cardinality $\kappa$.
\end{definition}

They proved that $\mathfrak{adp}\leq \mathfrak q_1$. It is also immediate from the separation axioms that $\mathfrak q_1\leq \mathfrak q_2\leq \mathfrak q_{2\frac{1}{2}}\leq \mathfrak q$; together with $\mathfrak q\leq\mathfrak b$, this gives, by Theorem~\ref{thm:adp-equals-dp},
\[
\mathfrak p\leq \mathfrak{dp}=\mathfrak{adp}\leq \mathfrak q_1\leq \mathfrak q_2\leq \mathfrak q_{2\frac{1}{2}}\leq \mathfrak q\leq \mathfrak b.
\]
Banakh and Bazylevych asked whether any of these inequalities between different versions of $\mathfrak q$ can be strict.
They also proved that every submetrizable space, that is, every space admitting a coarser metrizable topology, of cardinality less than $\mathfrak q$ is a $Q$-space.
Thus, any separation of these cardinals from $\mathfrak q$ must use non-submetrizable examples \cite{banakh2023q}.

More specifically, they asked whether $\mathfrak{ap}\leq \mathfrak q_2$.
We show in this section that the answer is positive.

We start by defining a cardinal $\mathfrak{at}$ that is a tree analogue of $\mathfrak{ap}$, which will be proven to be equal to $\mathfrak q_2$ in Theorem~\ref{thm:at-equals-q2}.

In this paper, an $\omega$-tree is a rooted tree of height $\omega$ with finite levels, and subtrees are assumed to be nonempty and downward closed.  For $m<\omega$, write
\[
  T_m=\{t\in T:\htt(t)=m\}
  \quad\text{and}\quad
  T_{<m}=\{t\in T:\htt(t)<m\}.
\]
Two subtrees $S,R$ of $T$ are almost disjoint if $|S\cap R|<\omega$.
A family of subtrees of $T$ is almost disjoint if any two distinct members are almost disjoint.
An infinite subtree means a subtree with infinitely many nodes.

Recall that a branch of a tree $T$ is a maximal linearly ordered subset of $T$.

\begin{definition}\label{def:at}
Let $\mathfrak{at}$ be the least cardinal $\kappa$ for which there is an
$\omega$-tree $T$, an almost disjoint family $\mathcal X$ of infinite subtrees of $T$ with $|\mathcal X|\leq \kappa$, and a subcollection
$\mathcal Y\subseteq\mathcal X$ such that $\mathcal Y$ cannot be weakly
separated from $\mathcal X\setminus\mathcal Y$.
Here a set $D\subseteq T$
weakly separates $(\mathcal Y,\mathcal X\setminus\mathcal Y)$ if
\[
  \forall Y\in\mathcal Y\ |D\cap Y|=\omega,
  \qquad \text{and}\qquad
  \forall Z\in\mathcal X\setminus\mathcal Y\ |D\cap Z|<\omega.
\]
\end{definition}

It follows immediately that $\mathfrak{ap}\leq \mathfrak{at}$, since every $\omega$-tree is countably infinite and, in the definition of $\mathfrak{ap}$, $\omega$ can be replaced by any infinite countable set.

\begin{lemma}\label{lem:tree-fsigma-separator}
Let $\mathcal A$ be an almost disjoint family of infinite subtrees of an
$\omega$-tree $T$, let $\mathcal Y\subseteq\mathcal A$, and put
$\mathcal Z=\mathcal A\setminus\mathcal Y$.
For each $Y\in\mathcal Y$, choose an infinite branch $b_Y\subseteq Y$, and put
\[
  H=\{b_Y:Y\in\mathcal Y\},\qquad X=H\cup\mathcal Z.
\]
The sets
\[
  N_F=\{S\in X:S\cap F=\emptyset\}\qquad(F\subseteq T\text{ finite})
\]
form a countable base for a Hausdorff topology $\tau^-$ on $X$.

If $|\mathcal Z|<\mathfrak b$ and $H$ is $F_\sigma$ in $(X,\tau^-)$,
then $\mathcal Y$ can be weakly separated from $\mathcal Z$.
\end{lemma}

\begin{proof}
The almost disjointness of $\mathcal A$ implies $b_Y\neq b_{Y'}$ whenever $Y, Y'$ are distinct members of $\mathcal Y$, and that $H\cap \mathcal Z=\emptyset$.
Thus, $X$ is an almost disjoint
family of infinite subtrees of $T$ and $|X|=|\mathcal A|$.

Notice that
$N_\emptyset=X$ and $N_F\cap N_G=N_{F\cup G}$.
Thus, these sets form a countable base for a topology on $X$.

We claim that this topology is Hausdorff.
For each $m \in \omega$, let $T_m$ be the $m$-th level of $T$.

Let $A,B\in X$ be distinct, and choose $m<\omega$ such that
$A\cap B\cap T_m=\emptyset$.
Then $N_{T_m\setminus A}$ and $N_{T_m\setminus B}$ are open
neighborhoods of $A$ and $B$, respectively.
If $C$ belonged to both neighborhoods, we would have
\[
  C\cap T_m\subseteq A\cap B\cap T_m=\emptyset.
\]
This is impossible by König's lemma.

For the second part, write
\[
  H=\bigcup_{n<\omega}H_n,\qquad S_n=\bigcup_{b\in H_n}b,
\]
where each $H_n$ is closed in $X$.
\begin{claim}
  For every $n<\omega$ and every $Z\in\mathcal Z$, the set $S_n\cap Z$ is finite.
\end{claim}
\begin{proof}[Proof of the claim]
Since $Z\notin H_n$ and $H_n$ is closed, there exists a finite $F\subseteq T$ such that
\[
  Z\in N_F\subseteq X\setminus H_n.
\]

If $F=\emptyset$, then $H_n=\emptyset$ and the claim is immediate.
Otherwise, let $M=\max\{\htt(t):t\in F\}+1$.

Given $t\in S_n\cap Z$, $\htt(t)<M$, for if $\htt(t)\geq M$, choose, as $t\in S_n$, an element $b\in H_n$
with $t\in b$.
As $b \in H_n$, it follows that $b\notin N_F$, so there exists $u\in b\cap F$.

Since $b$ is a branch, since $u, t\in b$, and since $\htt(t)\geq M>\htt(u)$, we have $u<t$.
As $t\in Z$ and $Z$ is a subtree, we have $u\in Z$, contradicting $Z\in N_F$.
Consequently,
\[
  S_n\cap Z\subseteq T_{<M},
\]
which is finite.
\end{proof}

For each $Z\in\mathcal Z$, choose $f_Z\in\omega^\omega$ such that
\[
  S_n\cap Z\subseteq T_{<f_Z(n)}
  \qquad(n<\omega).
\]
As $|\mathcal Z|<\mathfrak b$, there exists
$g\in\omega^\omega$ such that $f_Z\leq^*g$ for every
$Z\in\mathcal Z$. Define
\[
  D=\bigcup_{n<\omega}\bigl(S_n\setminus T_{<g(n)}\bigr).
\]

Given $Y\in\mathcal Y$, there exists $n<\omega$ such that $b_Y\in H_n$.
Hence $b_Y\setminus T_{<g(n)}\subseteq S_n\setminus T_{<g(n)}$, so $|D\cap Y|=\omega$.

Given $Z\in\mathcal Z$, choose $N<\omega$ such that
$f_Z(n)\leq g(n)$ for all $n\geq N$.
For such numbers $n$, the set $S_n\setminus T_{<g(n)}$ is disjoint from $Z$.
Therefore
\[
  D\cap Z\subseteq\bigcup_{n<N}(S_n\cap Z),
\]
a finite union of finite sets. Thus $|D\cap Z|<\omega$.
It follows that $D$ weakly separates $\mathcal Y$ from $\mathcal Z$, completing the proof.
\end{proof}

\begin{proposition}\label{prop:q2-below-at}
$\mathfrak q_2\leq \mathfrak{at}$.
\end{proposition}
\begin{proof}
Fix $\kappa<\mathfrak q_2$, an almost disjoint family $\mathcal A$ of
infinite subtrees of an $\omega$-tree $T$ with $|\mathcal A|\leq\kappa$,
and $\mathcal Y\subseteq\mathcal A$.
Choose branches by König's lemma and form $H$, $X$, and $\tau^-$ as in
Lemma~\ref{lem:tree-fsigma-separator}.
Since $(X,\tau^-)$ is second-countable and Hausdorff and
$|X|\leq\kappa<\mathfrak q_2$, it is a $Q$-space. In particular, $H$
is $F_\sigma$ in $\tau^-$.
Also $|\mathcal A\setminus\mathcal Y|\leq\kappa<\mathfrak q_2\leq\mathfrak b$,
so the lemma gives a weak separator of $\mathcal Y$ from
$\mathcal A\setminus\mathcal Y$. Thus $\kappa<\mathfrak{at}$.
\end{proof}

We will need the following auxiliary lemma.
\begin{lemma}\label{prop:at-le-b}
  $\mathfrak{at}\leq \mathfrak q$.
  Thus, $\mathfrak{at}\leq \mathfrak b$.
\end{lemma}
\begin{proof}
  Proposition~3.4 of \cite{brendle1999dow} says that:
  \[
    \mathfrak q = \min\left\{|\mathcal A|:
    \begin{array}{l}
      \mathcal A\subseteq [2^{<\omega}]^\omega\text{ is a family of branches, and}\\
      \exists\mathcal B\subseteq\mathcal A\text{ such that }(\mathcal A\setminus\mathcal B,\mathcal B)\\
      \text{is not weakly separated}
    \end{array}\right\}.
  \]

  As $2^{<\omega}$ is an $\omega$-tree and every branch is a subtree of $2^{<\omega}$, it follows that $\mathfrak{at}\leq \mathfrak q$.
  As $\mathfrak q\leq \mathfrak b$, we have $\mathfrak{at}\leq \mathfrak b$.
\end{proof}

The following bounding argument reduces $F_{\sigma\delta}$ sets of
cardinality less than $\mathfrak b$ to $F_\sigma$ sets.

\begin{lemma}\label{lem:diagonal-collapse}
Let $A$ be an $F_{\sigma\delta}$ subset of a topological space $X$.
If $|A|<\mathfrak b$, then $A$ is $F_\sigma$ in $X$.
\end{lemma}

\begin{proof}
Write
\[
  A=\bigcap_{n\in\omega}\bigcup_{m<\omega}F_{n,m},
\]
where the sets $F_{n,m}$ are closed and increasing in $m$.
Put $E_{n,m}=\bigcap_{i\leq n}F_{i,m}$ and
$H_n=\bigcup_{m\in\omega}E_{n,m}$.
Then $H_n=\bigcap_{i\leq n}\bigcup_{m\in\omega}F_{i,m}$, so the
sets $H_n$ are decreasing and $A=\bigcap_{n\in \omega} H_n$.
For each $a\in A$, choose $f_a\in\omega^\omega$ with
$a\in E_{n,f_a(n)}$ for all $n$.
Since $|A|<\mathfrak b$, there is $g\in\omega^\omega$ such that
$f_a\leq^*g$ for every $a\in A$. We have
\[
  A=\bigcup_{N\in\omega}\bigcap_{n\geq N}E_{n,g(n)}.
\]
The forward inclusion follows from eventual domination and monotonicity
in $m$. Conversely, a point in $\bigcap_{n\geq N}E_{n,g(n)}$ belongs
to every $H_n$ for $n\geq N$, and hence to $A$ because the $H_n$ are
decreasing. Each intersection on the right is closed, so $A$ is $F_\sigma$.
\end{proof}

Now we prove the main result of this section.

\begin{thm}\label{thm:at-equals-q2}
$\mathfrak{ap}\leq \mathfrak{at}=\mathfrak q_{2}$.
\end{thm}

\begin{proof}
We have already argued that $\mathfrak{ap}\leq \mathfrak{at}$.
By Proposition~\ref{prop:q2-below-at}, it suffices to show that $\mathfrak{at}\leq \mathfrak q_{2}$.
Let $\kappa<\mathfrak{at}$.
We show that $\kappa<\mathfrak q_{2}$.

Let $X$ be a second-countable Hausdorff space with $|X|\leq\kappa$.
We prove that $X$ is a $Q$-space.
We may assume that $X$ has at least two points.

Fix a countable base $\mathcal B$ and enumerate all ordered
pairs of disjoint members of $\mathcal B$ as $(U_k,V_k)_{k<\omega}$,
repeating pairs if necessary.
For each $k<\omega$, let
\[
  C_k^0=X\setminus U_k,\qquad C_k^1=X\setminus V_k.
\]
These sets are closed and satisfy $C_k^0\cup C_k^1=X$.
For each $s\in 2^{<\omega}$, define
\[
  F_s=\bigcap_{k<|s|}C_k^{s(k)},\qquad F_\emptyset=X,
\]
and, for each $x\in X$, put
\[
  T_x=\{s\in 2^{<\omega}:x\in F_s\}.
\]
\begin{claim}
  The collection $\mathcal T=\{T_x:x\in X\}$ is an almost disjoint family of infinite subtrees of $2^{<\omega}$.
\end{claim}
\begin{proof}[Proof of the claim]
  Every $F_s$ is closed, and $F_t\subseteq F_s$ whenever $s\subseteq t$.
  Thus each $T_x$ is a subtree of $2^{<\omega}$.
  Since each pair $\{C_k^0,C_k^1\}$ covers $X$, the tree $T_x$ meets
  every level of $2^{<\omega}$, so it is infinite.

  Now fix two distinct points $x,y\in X$.
  We show that $T_x\cap T_y$ is finite.

  As $X$ is Hausdorff, there exists $k\in \omega$ such that $x\in U_k$, $y\in V_k$, and $U_k\cap V_k=\emptyset$.
  Thus, if $s\in 2^{\geq k+1}$, then $s(k)=0$ implies $x\notin F_s$, and $s(k)=1$ implies $y\notin F_s$.
  Consequently, $T_x\cap T_y\subseteq2^{\leq k}$, which is finite.
\end{proof}

Notice that $|\mathcal T|=|X|\leq\kappa<\mathfrak{at}$.
Now fix $A\subseteq X$.
We will show that $A$ is an $F_\sigma$ subset of $X$.

By the definition of $\mathfrak{at}$, there exists
$D\subseteq 2^{<\omega}$ such that, for every $x\in A$ and every $y\in X\setminus A$,
\[
  |D\cap T_x|=\omega \qquad \text{and} \qquad
  |D\cap T_y|<\omega.
\]
For $(n,m)\in\omega^2$, define the closed set
\[
  E_{n,m}=\bigcup\{F_s:s\in D,\ n\leq |s|\leq m\}.
\]
Since the levels of $2^{<\omega}$ are finite, each $E_{n,m}$ is a finite
union of closed sets. Moreover,
\[
  A=\bigcap_{n<\omega}\bigcup_{m<\omega}E_{n,m}.
\]
Indeed, for $x\in A$, the infinite set $D\cap T_x$ contains nodes of
arbitrarily large length, so for every $n$ there is $m$ with
$x\in E_{n,m}$. Conversely, if $y\notin A$, the set $D\cap T_y$ is
finite. Choose $n=\sup\{|s|:s\in D\cap T_y\}+1$.
Then
$y\notin E_{n,m}$ for every $m\in \omega$.
Thus $A$ is $F_{\sigma\delta}$ in $X$.
Since $|A|\leq\kappa<\mathfrak{at}\leq\mathfrak b$ by
Lemma~\ref{prop:at-le-b}, Lemma~\ref{lem:diagonal-collapse} implies
that $A$ is $F_\sigma$ in $X$.

As $A$ is an arbitrary subset of $X$, the space $X$ is a $Q$-space.
As $X$ was arbitrary, $\kappa<\mathfrak q_2$.
Therefore $\mathfrak{at}\leq\mathfrak q_2$, completing the proof.
\end{proof}

%% file: sections/diagonal.tex
\section{$G_\delta$ diagonals}
\label{sec:small-diagonals}

Theorem~\ref{thm:at-equals-q2} establishes
$\mathfrak{ap}\leq\mathfrak{at}=\mathfrak q_2$.
We now show that the same cardinal is obtained by replacing the Hausdorff
condition with the $T_1$ axiom together with a $G_\delta$-diagonal.

\begin{definition}\label{def:q-diagonal-variants}
For $i\in\{1,2,2\frac{1}{2}\}$, let $\mathfrak q_{i\Delta}$ be the least
cardinality of a second-countable $T_i$ non-$Q$-space with a
$G_\delta$-diagonal. Thus $\mathfrak q_{1\Delta}$,
$\mathfrak q_{2\Delta}$ and $\mathfrak q_{2\frac{1}{2}\Delta}$ require,
respectively, the $T_1$, Hausdorff and Urysohn separation axioms.
\end{definition}

Clearly, for each $i\in\{1,2,2\frac{1}{2}\}$, we have $\mathfrak q_i\leq\mathfrak q_{i\Delta}$.
Moreover, we have
\[
  \mathfrak q_{1\Delta}\leq\mathfrak q_{2\Delta}
  \leq\mathfrak q_{2\frac{1}{2}\Delta}\leq\mathfrak q\leq\mathfrak b.
\]

In this section we will show that $\mathfrak q_{1\Delta}=\mathfrak{at}=\mathfrak q_2$ and that $\mathfrak q_{2\frac{1}{2}\Delta}=\mathfrak q_{2\frac{1}{2}}$.

\begin{lemma}\label{lem:q2-leq-q1delta}
Every second-countable space $X$ with a $G_\delta$-diagonal and
$|X|<\mathfrak q_2$ is a $Q$-space. Consequently,
$\mathfrak q_2\leq\mathfrak q_{1\Delta}$.
\end{lemma}

\begin{proof}
Write
$\Delta_X=\bigcap_{n<\omega}W_n$, where each $W_n$ is open in $X^2$
and $W_{n+1}\subseteq W_n$.
Fix a countable base $\mathcal B$ for $X$. For each $n\in \omega$, let
\[
  \mathcal C_n=\{U\in \mathcal B:U\times U\subseteq W_n\}.
\]
Each $\mathcal C_n$ is an open cover of $X$.
Enumerate $\mathcal C_n=\{U_{n,i}:i<\omega\}$, repeating members
if necessary.
For each $x\in X$, choose $f_x(n)$ with $x\in U_{n,f_x(n)}$.
Since $|X|<\mathfrak q_2\leq\mathfrak q\leq\mathfrak b$,
choose $g\in\omega^\omega$ such that for every $x \in X$, $f_x\leq^*g$. For $N<\omega$, put
\begin{equation}\label{eq:diagonal-finite-cover-pieces}
  X_N=\bigcap_{n\geq N}\bigcup_{i\leq g(n)}U_{n,i}.
\end{equation}
Each $X_N$ is $G_\delta$ in $X$, and $X=\bigcup_{N\in \omega} X_N$.

Fix $N\in\omega$ and write
$V_{n,i}=X_N\cap U_{n,i}$ for $n\geq N$ and $i\leq g(n)$.
Let $\tau_N$ be the original subspace topology on $X_N$, and let
$\sigma_N$ be the topology generated by the sets $X_N\setminus V_{n,i}$.
Their finite intersections form a countable base of $\tau_N$-closed sets.

We claim that $\sigma_N$ is Hausdorff. Given distinct $x,y\in X_N$,
choose $n\geq N$ with $(x,y)\notin W_n$.
Since $U_{n,i}\times U_{n,i}\subseteq W_n$, no $V_{n,i}$ contains both points.
The sets
\[
  H_x=\bigcap_{\substack{i\leq g(n)\\x\notin V_{n,i}}}
       (X_N\setminus V_{n,i}),
  \qquad
  H_y=\bigcap_{\substack{i\leq g(n)\\y\notin V_{n,i}}}
       (X_N\setminus V_{n,i})
\]
are $\sigma_N$-open neighborhoods of $x$ and $y$, respectively.
If $z\in H_x\cap H_y$, choose $i\leq g(n)$ with $z\in V_{n,i}$,
using \eqref{eq:diagonal-finite-cover-pieces}.
The definitions of $H_x$ and $H_y$ then imply $x,y\in V_{n,i}$,
a contradiction. Thus $H_x\cap H_y=\emptyset$.

Fix $A\subseteq X$. Since $|X_N|<\mathfrak q_2$, the space
$(X_N,\sigma_N)$ is a $Q$-space, so $A\cap X_N$ is $G_\delta$
in $\sigma_N$. Every $\sigma_N$-open set is a countable union of
$\tau_N$-closed basic sets. Hence $A\cap X_N$ is $F_{\sigma\delta}$
in $\tau_N$, and we may choose an $F_{\sigma\delta}$ set $B_N$ in $X$
such that $A\cap X_N=B_N\cap X_N$. Since the sets $X_N$ cover $X$, we have
\[
  A=\bigcap_{N\in\omega}\bigl((X\setminus X_N)\cup B_N\bigr).
\]
Each $X\setminus X_N$ is $F_\sigma$, so this expression shows directly
that $A$ is $F_{\sigma\delta}$ in $X$. As $|A|<\mathfrak b$,
Lemma~\ref{lem:diagonal-collapse} implies that $A$ is $F_\sigma$.
Since $A$ was arbitrary, $X$ is a $Q$-space.
\end{proof}

\begin{thm}\label{thm:q1delta-equals-at}
$\mathfrak q_{1\Delta}=\mathfrak{at}=\mathfrak q_2$.
\end{thm}

\begin{proof}
By Lemma~\ref{lem:q2-leq-q1delta} and Theorem~\ref{thm:at-equals-q2},
it remains only to show that $\mathfrak q_{1\Delta}\leq\mathfrak{at}$.

Fix $\kappa<\mathfrak q_{1\Delta}$, an almost disjoint family
$\mathcal A$ of infinite subtrees of an $\omega$-tree $T$ with
$|\mathcal A|\leq\kappa$, and a subfamily $\mathcal Y\subseteq\mathcal A$.
Put $\mathcal Z=\mathcal A\setminus\mathcal Y$.

Choose branches by König's lemma and form $H$, $X$, and $\tau^-$ as in
Lemma~\ref{lem:tree-fsigma-separator}. On the same set $X$, consider
the topology $\tau^+$ generated by
\[
  P_t=\{S\in X:t\in S\}\qquad(t\in T).
\]
Their finite intersections form a countable base. If $S, R\in X$ are distinct,
let $t\in S\setminus R$.
It follows that $S\in P_t$ and $R\notin P_t$, so $(X,\tau^+)$ is $T_1$.
Moreover,
\begin{equation}\label{eq:positive-tree-diagonal}
  \Delta_X
  =\bigcap_{n<\omega}\bigcup_{t\in T_n}(P_t\times P_t).
\end{equation}
Indeed, given $S \in X$, for every $n \in \omega$ there exists $t \in S \cap T_n$,
and then $S \in P_t$, so $(S,S) \in P_t \times P_t$.
Conversely, if $S$ and $R$ are distinct elements of $X$, let $n$ be such that $S\cap R\cap T_n=\emptyset$.
Then for no $t\in T_n$ do we have $(S,R)\in P_t\times P_t$.
Thus, $(X,\tau^+)$ has a $G_\delta$-diagonal.

Since $|X|\leq\kappa<\mathfrak q_{1\Delta}$, the space $(X,\tau^+)$
is a $Q$-space, so $H$ is $G_\delta$ in $\tau^+$.
Each $P_t$ is $\tau^-$-closed, because $X\setminus P_t=N_{\{t\}}$.
Hence every $\tau^+$-open set is $F_\sigma$ in $\tau^-$, and $H$ is
$F_{\sigma\delta}$ in $\tau^-$.
As $|H|\leq\kappa<\mathfrak q_{1\Delta}\leq\mathfrak b$,
Lemma~\ref{lem:diagonal-collapse} implies that $H$ is an $F_\sigma$ set in $\tau^-$.
Since also $|\mathcal Z|<\mathfrak b$, Lemma~\ref{lem:tree-fsigma-separator}
grants a weak separator of $\mathcal Y$ from $\mathcal Z$.
Thus $\kappa<\mathfrak{at}$, proving the remaining inequality.
\end{proof}

We next identify the Urysohn $G_\delta$-diagonal cardinal with the Urysohn cardinal.
The first part of the following lemma is a folklore result, but we include a proof for completeness.

\begin{lemma}\label{lem:second-countable-urysohn-gdelta-diagonal}
Every second-countable Urysohn space has a $G_\delta$-diagonal.
Consequently,
$\mathfrak q_{2\frac{1}{2}\Delta}=\mathfrak q_{2\frac{1}{2}}$.
\end{lemma}

\begin{proof}
We may assume that the space $X$ has at least two points.

Fix a countable base $\mathcal B$ for a second-countable Urysohn
space $X$. For every pair $U,V\in\mathcal B$ with
$\overline U\cap\overline V=\emptyset$, the set
\[
  \mathcal U_{U,V}=\{X\setminus\overline U,\ X\setminus\overline V\}
\]
is an open cover of $X$.
Let $\mathcal V=\{\mathcal U_{U,V}: U,V\in\mathcal B, \overline U\cap\overline V=\emptyset\}$, which is countable.
As $X$ is Urysohn and has at least two points, $\mathcal V\neq\emptyset$.

We show that:
\begin{equation}\label{eq:urysohn-gdelta-diagonal}
  \Delta_X=\bigcap_{\mathcal U\in\mathcal V}
    \bigcup_{O\in\mathcal U}(O\times O).
\end{equation}

If $(x,x)\in\Delta_X$, then for every $\mathcal U\in\mathcal V$ there is $O\in\mathcal U$ with $x\in O$, so $(x,x)\in O\times O$.

Conversely, given distinct $x,y\in X$, by $T_{2\frac{1}{2}}$, there exist $U,V\in\mathcal B$ such that $x\in U$, $y\in V$, and $\overline U\cap\overline V=\emptyset$. No member of $\mathcal U_{U,V}$ contains both $x$ and $y$.
Thus $(x, y)$ is not in the right side of \eqref{eq:urysohn-gdelta-diagonal}, and the equality is established.
\end{proof}

%% file: sections/forcing.tex
\section{A model with $\mathfrak{ap}<\mathfrak{at}$}
\label{sec:ap-below-at-model}

In this section we prove a consistency result separating $\mathfrak{ap}$ from
$\mathfrak{at}$, and therefore from $\mathfrak q_{2}$.
The construction is based on Brendle's model for
$\mathfrak{ap}<\mathfrak q$ \cite{brendle1999dow}.

We use the forcing convention that $q\leq p$ means that $q$ is stronger than
$p$.  A subset of a forcing preorder is centered if every finite subfamily has
a common extension; a forcing is $\sigma$-centered if it is the union of
countably many centered sets.  The countable chain condition (ccc) means that
every antichain is countable.  If $\mathbb P\subseteq\mathbb Q$ are forcing
preorders, then $\mathbb P$ is a complete suborder of $\mathbb Q$ if every
maximal antichain of $\mathbb P$ remains maximal in $\mathbb Q$.

Following \cite{kunen2011set}, if $\mathbb P$ is a forcing preorder and $\tau$ is a $\mathbb P$-name, a $\mathbb P$-nice name for a subset of $\tau$ is a $\mathbb P$-name of the form $\bigcup_{\sigma \in \dom \tau}\{\sigma\}\times A_\sigma$ for some family $(A_\sigma:\sigma\in\dom \tau)$ of antichains of $\mathbb P$.

For undefined forcing concepts and standard facts about forcing, we refer the reader to \cite{kunen2011set}.

\subsection{Intertwined families}
\label{subsec:intertwined-families}

The following notion, due to
Brendle \cite{brendle1999dow}, is the small almost disjoint configuration that will be preserved
throughout the iteration.
\begin{definition}\label{def:intertwined}
A pair $\langle\mathcal B,\mathcal C\rangle$ of disjoint subcollections of
$[\omega]^\omega$, each of cardinality $\aleph_1$, is \emph{intertwined} if
$\mathcal B\cup\mathcal C$ is an almost disjoint family and, whenever
$E\in[\omega]^\omega$ has infinite intersection with uncountably many members
of $\mathcal C$, then $E$ has infinite intersection with all but countably many
members of $\mathcal B$.
\end{definition}

The role of intertwined families is to keep $\mathfrak{ap}$ small throughout the forcing
iteration.  Indeed, if such a family survives to the final extension, then
$\mathfrak{ap}=\omega_1$: the family $\mathcal B\cup\mathcal C$ is an almost
disjoint family of size $\omega_1$, and the defining implication for intertwinedness
prevents $\mathcal C$ from being weakly separated from $\mathcal B$.

\subsection{Forcing blocks}
\label{subsec:forcing-blocks}

In this section, we define the forcing posets that will be used as building blocks for the iteration.

The forcing used here is the tree analogue of Brendle's forcing
$Q(\mathcal A)$ from \cite{brendle1999dow}.  In Brendle's argument, conditions
choose a finite increasing sequence of natural numbers while avoiding finitely
many members of an almost disjoint family $\mathcal A$ of branches of $2^{<\omega}$.  Here the
ambient object is not the fixed tree $2^{<\omega}$ and the members of the almost
disjoint family are not necessarily branches; instead, it is an arbitrary
finitely branching $\omega$-tree $T$, and $\mathcal A$ is an almost disjoint family of
infinite subtrees of $T$. The generic object is a subset of $T$
which is almost disjoint from every member of $\mathcal A$, but still meets
infinitely every ground-model set outside the ideal generated by $\mathcal A$.

For an $\omega$-tree $T$, write
\[
  T^{\uparrow<\omega}
  =
  \{\sigma\in T^{<\omega}:
    i<j<|\sigma|\Rightarrow
    \htt(\sigma(i))<\htt(\sigma(j))\}.
\]
Thus $T^{\uparrow<\omega}$ consists of finite sequences of nodes whose heights
strictly increase; no compatibility in the tree order is required.  The
following forcing is a tree version of Brendle's forcing.

\begin{definition}\label{def:QT}
Let $T$ be an $\omega$-tree and let $\mathcal X$ be an almost disjoint family of subtrees of $T$.
The forcing poset $Q_T(\mathcal X)$ consists of triples $p=(\sigma_p,h_p,\mathcal X_p)$ with
\[
  \sigma_p\in T^{\uparrow<\omega},\qquad
  h_p:T^{\uparrow<\omega}\to\omega,
  \qquad
  \mathcal X_p\in[\mathcal X]^{<\omega},
\]
such that, for every $i<|\sigma_p|$, $h_p(\sigma_p\restriction i)\leq\htt(\sigma_p(i))$.

For $p, q \in Q_T(\mathcal X)$, we define
$q\leq p$ if, and only if:
\begin{enumerate}[label=(\arabic*)]
    \item $\sigma_p\subseteq \sigma_q$;
    \item $h_q\geq h_p$ pointwise;
    \item $\mathcal X_q\supseteq\mathcal X_p$;
    \item $\sigma_q(i)\notin\bigcup \mathcal X_p$ for all
    $|\sigma_p|\leq i<|\sigma_q|$.
\end{enumerate}
Given $p\in Q_T(\mathcal X)$, the \emph{stem} of $p$ is $\sigma_p$ and the \emph{side condition} of $p$ is $\mathcal X_p$.
\end{definition}

In the notation above, the stem $\sigma_p$ is a finite initial approximation to a generic set of nodes, the
function $h_p$ imposes height lower bounds for future extensions, and the side
condition $\mathcal X_p$ lists the subtrees that future nodes must avoid.

We first record that these forcing posets are $\sigma$-centered.  In particular,
they have the countable chain condition and preserve cardinals and cofinalities.
\begin{lemma}\label{lem:QT-sigma-centered}
Let $T$ be an $\omega$-tree and let $\mathcal X$ be an almost disjoint family of subtrees of $T$. Then:
\begin{enumerate}[label=(\arabic*)]
    \item if $p,q\in Q_T(\mathcal X)$ have the same stem, then there exists
    $r\in Q_T(\mathcal X)$ such that $r\leq p,q$ and $r$ has the same stem as
    $p$ and $q$;
    \item $Q_T(\mathcal X)$ is $\sigma$-centered.
\end{enumerate}
\end{lemma}

\begin{proof}
For (1), let $p=(\sigma,h_p,\mathcal X_p)$ and
$q=(\sigma,h_q,\mathcal X_q)$.
Define $r=(\sigma,h_r,\mathcal X_r)$, where
$\mathcal X_r=\mathcal X_p\cup\mathcal X_q$ and $h_r(\tau)=\max\{h_p(\tau),h_q(\tau)\}$ for every $\tau\in T^{\uparrow<\omega}$.
Then $r\leq p,q$ and $r$ has stem $\sigma$.

For (2), for each $\sigma\in T^{\uparrow<\omega}$ let $C_\sigma=\{p\in Q_T(\mathcal X):\text{the stem of }p\text{ is }\sigma\}$.
By (1) and a straightforward induction, each $C_\sigma$ is centered.
Since $T$ is countable, $T^{\uparrow<\omega}$ is countable, and $Q_T(\mathcal X)=\bigcup_{\sigma\in T^{\uparrow<\omega}}C_\sigma$.
Thus, $Q_T(\mathcal X)$ is $\sigma$-centered.
\end{proof}

Given an almost disjoint family $\mathcal X$ of subsets of a countable set $T$,
the \emph{ideal generated by $\mathcal X$ on $T$} is defined as follows.
\[
  \mathcal I_T(\mathcal X)
  =
  \{X\subseteq T:
    \exists\mathcal F\in[\mathcal X]^{<\omega}
    |X\setminus\bigcup\mathcal F|<\omega\}.
\]

We say that the ideal $\mathcal I_T(\mathcal X)$ is \emph{proper} if $T\notin\mathcal I_T(\mathcal X)$.
Notice that if $\mathcal X$ is infinite, then $\mathcal I_T(\mathcal X)$ is proper.

The following lemma records the basic separation property of $Q_T(\mathcal X)$.
Provided the ideal generated by $\mathcal X$ is proper, the generic set is
almost disjoint from every member of $\mathcal X$ and meets every
$\mathcal I_T(\mathcal X)$-positive set infinitely often.

\begin{lemma}\label{lem:QT-adds-separator}
Let $T$ be an $\omega$-tree and let
$\mathcal X$ be an almost disjoint family of subtrees of $T$.
Assume that $\mathcal I_T(\mathcal X)$ is proper.
Let $\dot D$ be a $Q_T(\mathcal X)$-name such that
\[
  Q_T(\mathcal X)\Vdash\dot D=\{t\in \check T:\exists p\in\dot G\ \exists i<|\sigma_p|\ (t=\sigma_p(i))\}.
\]
Then:
\begin{enumerate}[label=(\arabic*)]
    \item for every $A \in \mathcal X$, $Q_T(\mathcal X)\Vdash |\dot D\cap \check A|<\omega$;
    \item for every $B \in [T]^\omega\setminus \mathcal I_T(\mathcal X)$, $Q_T(\mathcal X)\Vdash |\dot D\cap \check B|=\omega$.
\end{enumerate}
\end{lemma}

\begin{proof}
For (1), fix $A \in \mathcal X$.  The set
$\mathcal D_A=\{p\in Q_T(\mathcal X):A\in\mathcal X_p\}$ is dense in
$Q_T(\mathcal X)$.
Working in a countable transitive model $M$, let $G$ be $Q_T(\mathcal X)$-generic over $M$.
Let $p \in G\cap \mathcal D_A$.
We claim that $\dot D_G\cap A\subseteq \{\sigma_p(i):i<|\sigma_p|\}$.

Given $t \in \dot D_G\cap A$, there are $q \in G$ and $i<|\sigma_q|$ such that $t=\sigma_q(i)$.
There exists $r \in G$ such that $r \leq p,q$.
If $i\geq|\sigma_p|$, then
$t=\sigma_q(i)=\sigma_r(i)\notin \bigcup \mathcal X_p$, contradicting
$t\in A\subseteq\bigcup\mathcal X_p$.  Hence $i<|\sigma_p|$.

For (2), fix $B\in[T]^\omega\setminus \mathcal I_T(\mathcal X)$ and $m\in\omega$, and define
\[
\mathcal D_{B,m}=\{p\in Q_T(\mathcal X):\exists i<|\sigma_p|(\sigma_p(i)\in B \land \htt(\sigma_p(i))\geq m)\}.
\]

Then $\mathcal D_{B,m}$ is dense in $Q_T(\mathcal X)$: given $p \in Q_T(\mathcal X)$, let $\bar m\in \omega$ be such that
\[
\bar m\geq \max(\{m, h_p(\sigma_p)\}\cup\{\htt(\sigma_p(i)):i<|\sigma_p|\})+1.
\]
As $B\notin\mathcal I_T(\mathcal X)$, $B\setminus \bigcup\mathcal X_p$ is infinite.
Since every level of $T$ is finite, $T_{<\bar m}$ is finite.
Thus, there is $t\in B\setminus \bigcup\mathcal X_p$ such that $\htt(t)\geq \bar m$.
Then $q=(\sigma_p^\frown\langle t\rangle,h_p,\mathcal X_p)$ is a condition extending $p$ and belonging to $\mathcal D_{B,m}$.

Now, given $B$ and $m$, every $p \in \mathcal D_{B,m}$ forces
$\dot D\cap \check B\not\subseteq\check{T_{<m}}$.
As $\mathcal D_{B,m}$ is dense, this means that $Q_T(\mathcal X)\Vdash \dot D\cap \check B\not\subseteq \check{T_{<m}}$.
Since this holds for every $m \in \omega$ and each $T_{<m}$ is finite, we have $Q_T(\mathcal X)\Vdash |\dot D\cap \check B|=\omega$.
\end{proof}

If $T$ has height $\omega$, then the infinite branches of $T$ are exactly the
chains that meet every level of $T$.
Such a branch meets each level in exactly one node.

We will use the following notation.

\begin{definition}\label{def:QT-extensible}
  Let $T$ be an $\omega$-tree, and let $\mathcal X$ be an almost disjoint
  family of subtrees of $T$.
  Let $p\in Q_T(\mathcal X)$ and let
  $\tau\in T^{\uparrow<\omega}$.
  We say that $p$ is extensible to $\tau$ if $\sigma_p\subseteq\tau$ and, for every
  $|\sigma_p|\leq i<|\tau|$,
  \[
    h_p(\tau\restriction i)\leq\htt(\tau(i))
    \quad\text{and}\quad
    \tau(i)\notin\bigcup \mathcal X_p.
  \]
\end{definition}
Thus, $p=(\sigma_p,h_p,\mathcal X_p)$ is extensible to $\tau$ exactly when
one can replace the stem of $p$ by $\tau$, keeping the same height function and
side condition, and obtain an extension of $p$.

\begin{lemma}\label{lem:QT-preparation}
Let $T$ be an $\omega$-tree, let $\mathcal X$ be an almost disjoint family of
subtrees of $T$, and let
\[
  q_\beta=(\sigma,h_\beta,\mathcal R\cup\mathcal S_\beta)
  \qquad(\beta\in I)
\]
be an uncountable family of conditions in $Q_T(\mathcal X)$ with the same stem $\sigma$.
Assume that the finite sets $\mathcal S_\beta$ are pairwise
disjoint and disjoint from $\mathcal R$.
Let $h':T^{\uparrow<\omega}\to\omega$ be such that $(\sigma,h',\mathcal R)$ is a condition and $h_\beta\geq h'$ for every $\beta\in I$.
Then there is $h\geq h'$ such that $q=(\sigma,h,\mathcal R)$ is a condition
and, for every $\tau\in T^{\uparrow<\omega}$ to which $q$ is extensible, the set
\[
  \{\beta\in I:q_\beta\text{ is extensible to }\tau\}
\]
is uncountable.
\end{lemma}

\begin{proof}
For each $n<\omega$, put
$L_n=\{\tau\in T^{\uparrow<\omega}:|\tau|=n\}$ and
$L_{<n}=\bigcup_{m<n}L_m$.
We define a sequence of functions $(g_n:n<\omega)$ such that, for every
$n<\omega$:
\begin{enumerate}[label=(\arabic*)]
    \item $g_n:T^{\uparrow<\omega}\to\omega$;
    \item $g_n|(T^{\uparrow<\omega}\setminus L_{<n})=h'|(T^{\uparrow<\omega}\setminus L_{<n})$;
    \item for all $m<n$, $g_n|L_{<m}=g_m|L_{<m}$ and $g_m\leq g_n$ pointwise;
    \item $(\sigma, g_n,\mathcal R)\in Q_T(\mathcal X)$;
    \item for all $\tau \in L_n$, if $(\sigma,g_n,\mathcal R)$ is not extensible to $\tau$, then $g_{n+1}(\tau)=h'(\tau)$;
    \item for all $\tau \in L_n$, if $(\sigma,g_n,\mathcal R)$ is extensible to $\tau$, then for every $t\in T$ such that
    $\tau^\frown\langle t\rangle\in T^{\uparrow<\omega}$,
    $t\notin\bigcup\mathcal R$, and
    $\htt(t)\geq g_{n+1}(\tau)$, the set
    $\{\beta\in I:q_\beta\text{ is extensible to }\tau^\frown\langle t\rangle\}$ is uncountable.
\end{enumerate}

We proceed by recursion on $n$.
For $n=0$, define $g_0=h'$.  Then (1)--(4) hold, and there is nothing to check for (5) and (6).

Assume that $(g_m:m\leq n)$ has been defined so that conditions (1)--(6) hold up to this point.
We define $g_{n+1}$ as follows.
For $\tau \in L_{<n}$, put $g_{n+1}(\tau)=g_n(\tau)$.
For $\tau \in T^{\uparrow<\omega}\setminus L_{<n+1}$, put $g_{n+1}(\tau)=h'(\tau)$.

Now fix $\tau\in L_n$.  If $(\sigma,g_n,\mathcal R)$ is not extensible to
$\tau$, put $g_{n+1}(\tau)=h'(\tau)$.  Otherwise, we first show that
\[
  I_\tau=\{\beta\in I:q_\beta\text{ is extensible to }\tau\}
\]
is uncountable.  If $\tau=\sigma$, this is immediate.  If
$\tau\supsetneq\sigma$, write $\tau=\eta^\frown\langle t\rangle$, where
$\eta\in L_{n-1}$.  Since $(\sigma,g_n,\mathcal R)$ is extensible to $\tau$,
we have
\[
  t\notin\bigcup\mathcal R
  \qquad\text{and}\qquad
  \htt(t)\geq g_n(\eta).
\]
Moreover, $(\sigma,g_{n-1},\mathcal R)$ is extensible to $\eta$, because
$g_{n-1}$ and $g_n$ agree on all proper initial segments of $\eta$.
By the induction hypothesis, applied to $\eta\in L_{n-1}$, the set
$\{\beta\in I:q_\beta\text{ is extensible to }\eta^\frown\langle t\rangle\}$
is uncountable. 
Hence, $I_\tau$ is uncountable.

Since the map $\beta\in I_\tau\longmapsto h_\beta(\tau)\in\omega$
has countable range, there is an uncountable set $J_\tau\subseteq I_\tau$ and a number
$j_0<\omega$ such that $h_\beta(\tau)\leq j_0$ for every $\beta\in J_\tau$.

For $\beta\in J_\tau$, put
\[
  U_\beta=\bigcup\mathcal S_\beta.
\]
The family $(U_\beta:\beta\in J_\tau)$ is pairwise almost disjoint, since the
$\mathcal S_\beta$'s are pairwise disjoint finite subsets of the almost
disjoint family $\mathcal X$.
Let
\[
  E_\tau=
  \{t\in T:\{\beta\in J_\tau:t\notin U_\beta\}\text{ is countable}\}.
\]

Thus, $E_\tau$ is the set of nodes which are forbidden by all but countably many
of the side conditions $\mathcal S_\beta$.

We claim that $E_\tau$ is finite.
For each $t\in E_\tau$, the set $\{\beta\in J_\tau:t\in U_\beta\}$
is cocountable in $J_\tau$.
Since $E_\tau$ is countable, $\bigcap_{t\in E_\tau}\{\beta\in J_\tau:t\in U_\beta\}$ is cocountable in $J_\tau$.
Choose distinct
$\beta,\gamma$ in this intersection.
Then $E_\tau\subseteq U_\beta\cap U_\gamma$.
Since $U_\beta$ and $U_\gamma$ are almost disjoint, $E_\tau$ is finite, as claimed.
Let
\[
  g_{n+1}(\tau)=\max(\{h'(\tau), j_0\}\cup\{\htt(t):t\in E_\tau\})+1.
\]
This completes the definition of $g_{n+1}$.

We now check the requirements.  Requirements (1)--(3) are immediate from the
definition of $g_{n+1}$ and the induction hypothesis.  Requirement (4) also
holds.
Indeed, let $i<|\sigma|$ and put $\eta=\sigma\restriction i$.
We must show that
\[
  g_{n+1}(\eta)\leq \htt(\sigma(i)).
\]
If $|\eta|<n$, then $g_{n+1}(\eta)=g_n(\eta)$, and the inequality follows from
$(\sigma,g_n,\mathcal R)\in Q_T(\mathcal X)$.
If $|\eta|>n$, then $g_{n+1}(\eta)=h'(\eta)$, and the inequality follows from
$(\sigma,h',\mathcal R)\in Q_T(\mathcal X)$.
Finally, if $|\eta|=n$, then $\eta$ is a proper initial segment of $\sigma$,
so $(\sigma,g_n,\mathcal R)$ is not extensible to $\eta$. Hence
$g_{n+1}(\eta)=h'(\eta)$, and the inequality again follows from
$(\sigma,h',\mathcal R)\in Q_T(\mathcal X)$.
Thus, $(\sigma,g_{n+1},\mathcal R)\in Q_T(\mathcal X)$.

Requirement (5) was arranged by definition.  For requirement (6), let
$\tau\in L_n$ be such that $(\sigma,g_n,\mathcal R)$ is extensible to $\tau$,
and let $t\in T$ satisfy
\[
  \tau^\frown\langle t\rangle\in T^{\uparrow<\omega},
  \qquad
  t\notin\bigcup\mathcal R,
  \qquad
  \htt(t)\geq g_{n+1}(\tau).
\]
Then $t\notin E_\tau$, and therefore $K_t=\{\beta\in J_\tau:t\notin U_\beta\}$ is uncountable.
For every $\beta\in K_t\subseteq I_\tau$, we have $q_\beta$ extensible to
$\tau$,
\[
  h_\beta(\tau)\leq j_0\leq g_{n+1}(\tau)\leq\htt(t) \qquad\text{and}\qquad
    t\notin\bigcup(\mathcal R\cup\mathcal S_\beta).
\]
Hence, $q_\beta$ is extensible to $\tau^\frown\langle t\rangle$.
Therefore, $K_t\subseteq\{\beta\in I:q_\beta\text{ is extensible to }\tau^\frown\langle t\rangle\}$, and the latter set is uncountable, concluding the verification of (6).
This completes the recursive construction.

Finally, define
\[
  h(\tau)=g_{n+1}(\tau)
  \qquad\text{whenever }\tau\in L_n.
\]

By construction, for every $\tau\in T^{\uparrow<\omega}$, $h(\tau)=g_{n+1}(\tau)\geq g_0(\tau)=h'(\tau)$, where $n=|\tau|$.

We claim that $q=(\sigma,h,\mathcal R)\in Q_T(\mathcal X)$.
Let $i<|\sigma|$.
Since $\sigma\restriction i$ is a proper initial segment of $\sigma$,
$(\sigma,g_i,\mathcal R)$ is not extensible to $\sigma\restriction i$.
Hence, by requirement (5),
\[
  h(\sigma\restriction i)=g_{i+1}(\sigma\restriction i)
  =h'(\sigma\restriction i)
  \leq \htt(\sigma(i)).
\]
Therefore, $q\in Q_T(\mathcal X)$.

It remains to verify the desired property.  Let $\tau\in T^{\uparrow<\omega}$
be such that $q$ is extensible to $\tau$.  If $\tau=\sigma$, then every
$q_\beta$ is extensible to $\tau$.  Otherwise, write
$\tau=\eta^\frown\langle t\rangle$ and let $|\eta|=n$.  Since $q$ is
extensible to $\tau$, we have
\[
  t\notin\bigcup\mathcal R
  \qquad\text{and}\qquad
  \htt(t)\geq h(\eta)=g_{n+1}(\eta).
\]
Also, $(\sigma,g_n,\mathcal R)$ is extensible to $\eta$, because $g_n$ agrees
with $h$ on all proper initial segments of $\eta$.  By requirement (6), the
set
\[
  \{\beta\in I:q_\beta\text{ is extensible to }
      \eta^\frown\langle t\rangle\}
\]
is uncountable.  Since $\tau=\eta^\frown\langle t\rangle$, this concludes the
proof.
\end{proof}

\subsection{Ranks and ground-model approximations}
\label{subsec:ranks}

We now introduce the rank machinery used to analyze names for subsets of $\omega$.
This is the point at which the proof most closely follows Brendle's rank argument,
with natural numbers replaced by nodes of high enough height in $T$.

\begin{definition}\label{def:QT-ranks}
Let $T$ be an $\omega$-tree, let $\mathcal X$ be an almost disjoint family of
subtrees of $T$, let $\mathcal R\in[\mathcal X]^{<\omega}$, and let $\dot D$
be a $Q_T(\mathcal X)$-name for a subset of $\omega$.
We define:
\begin{enumerate}[label=(\arabic*)]
    \item for each $\eta\in T^{\uparrow<\omega}$,
    $\operatorname{Succ}_{\mathcal R}(\eta) =
    \{\eta^\frown\langle t\rangle:
    t\in T\setminus\bigcup\mathcal R
    \text{ and }
    \eta^\frown\langle t\rangle\in T^{\uparrow<\omega}\}$.
    \item for each $\ell<\omega$, $B_\ell=
    \{\eta\in T^{\uparrow<\omega}:\exists r\in Q_T(\mathcal X)\,
    (\sigma_r=\eta\text{ and }r\Vdash\ell\in\dot D)\}$.
\end{enumerate}
For each $\ell<\omega$, define, by recursion on $\alpha<\omega_1$, sets
$W^\ell_\alpha\subseteq T^{\uparrow<\omega}$ as follows:
\begin{itemize}
    \item $W^\ell_0=B_\ell$;
    \item for $0<\alpha<\omega_1$,
    $W^\ell_\alpha
    =
    \bigcup_{\xi<\alpha}W^\ell_\xi\cup
    \{\eta\in T^{\uparrow<\omega}:
      |\operatorname{Succ}_{\mathcal R}(\eta)
        \cap\bigcup_{\xi<\alpha}W^\ell_\xi|=\omega\}$.
\end{itemize}
Let $W^\ell=\bigcup_{\alpha<\omega_1}W^\ell_\alpha$.
For $\eta\in W^\ell$, the rank of $\eta$ with respect to $(\ell,\mathcal R,\dot D)$ is
\[
  \rho_\ell(\eta)
  =
  \min\{\alpha<\omega_1:\eta\in W^\ell_\alpha\}.
\]
If $\eta\in T^{\uparrow<\omega}\setminus W^\ell$, write $\rho_\ell(\eta)=\infty$.
Finally, for $\eta\in T^{\uparrow<\omega}$ and $i<\omega$, put
\[
D_{\eta,i}
=
\{\ell<\omega:
  \exists t\in T\,
  (
    \eta^\frown\langle t\rangle\in
    \operatorname{Succ}_{\mathcal R}(\eta)\cap W^\ell
    \text{ and }
    \htt(t)\geq i
  )\}.
\]
\end{definition}

The sets $W^\ell_\alpha$ measure how close a stem is to being able to force
$\ell\in\dot D$.
At level zero, we put all stems from which $\ell\in\dot D$
can already be forced.
At later stages, a stem is admitted once it has infinitely many one-step
extensions that avoid the fixed finite set $\mathcal R$ and were admitted earlier.
Thus $\rho_\ell(\eta)$ is the
first stage at which $\eta$ enters this inductive closure, while
$\rho_\ell(\eta)=\infty$ means that this process never reaches $\eta$.

The sets $D_{\eta,i}$ are ground-model approximations to the name
$\dot D$ above the stem $\eta$.
An integer $\ell$ belongs to $D_{\eta,i}$ if some allowed one-step extension
of $\eta$ by a node of height at least $i$ has countable $\ell$-rank.
These approximations are designed to reflect the possible decisions about
$\dot D$ made by extensions of the forcing condition in the ground model, and will be used in the proof of Lemma~\ref{lem:tree-preserves-intertwined}.

\begin{lemma}\label{lem:QT-rank-facts}
Let $T$ be an $\omega$-tree, let $\mathcal X$ be an almost disjoint family of
subtrees of $T$, and assume that $\mathcal I_T(\mathcal X)$ is proper.
Let $\mathcal R\in[\mathcal X]^{<\omega}$ and let $\dot D$ be a
$Q_T(\mathcal X)$-name for a subset of $\omega$.
Let $\rho_\ell=\rho_{\ell,\mathcal R,\dot D}$, $D_{\eta,i}$ and $W^\ell$ be as in Definition~\ref{def:QT-ranks}.
Then, for every $\eta\in T^{\uparrow<\omega}$ and every $\ell<\omega$,
\begin{enumerate}[label=(\arabic*)]
    \item if $\eta\in W^\ell$, then $\ell\in D_{\eta,i}$ for every
    $i<\omega$;

    \item if $(\mathcal S_\beta:\beta\in I)$ is a pairwise disjoint family of
    finite subfamilies of $\mathcal X$, then, for every
    $\eta\in T^{\uparrow<\omega}$ with $0<\rho_\ell(\eta)<\omega_1$, for all
    but at most one $\beta\in I$ the set
    \[
      \{t\in T\setminus\bigcup(\mathcal R\cup\mathcal S_\beta):
        \eta^\frown\langle t\rangle\in T^{\uparrow<\omega}
        \text{ and }
        \rho_\ell(\eta^\frown\langle t\rangle)<\rho_\ell(\eta)
      \}
    \]
    is infinite.
\end{enumerate}
\end{lemma}

\begin{proof}
For (1), fix $i<\omega$.

First assume that $0<\rho_\ell(\eta)<\omega_1$, and put
$\alpha=\rho_\ell(\eta)$.  Since $\alpha>0$ is the first stage at which
$\eta$ belongs to $W^\ell_\alpha$, the definition of $W^\ell_\alpha$ says that
\[
  L=
  \{t\in T\setminus\bigcup\mathcal R:
      \eta^\frown\langle t\rangle\in T^{\uparrow<\omega}
      \text{ and }
      \eta^\frown\langle t\rangle\in\bigcup_{\xi<\alpha}W^\ell_\xi\}
\]
is infinite.  Since the levels of $T$ are finite, $L$ contains nodes of
arbitrarily large height.  Choose $t\in L$ with
$\htt(t)\geq i$.  Then
$\eta^\frown\langle t\rangle\in W^\ell$, $t\notin\bigcup\mathcal R$, and
$\htt(t)\geq i$, so $\ell\in D_{\eta,i}$.

Now assume that $\rho_\ell(\eta)=0$.  Then $\eta\in B_\ell$, so there is a
condition $r\in Q_T(\mathcal X)$ such that $\sigma_r=\eta$ and
$r\Vdash \ell\in\dot D$.
Let
\[
  \bar m>\max\bigl(\{i,h_r(\eta)\}\cup
      \{\htt(\eta(j)):j<|\eta|\}\bigr).
\]
Since $\mathcal R\cup\mathcal X_r$ is a finite subset of $\mathcal X$ and
$\mathcal I_T(\mathcal X)$ is proper,
$T\setminus\bigcup(\mathcal R\cup\mathcal X_r)$ is infinite, and, because the
levels of $T$ are finite, it has nodes of arbitrarily large height.
Choose $t\in T\setminus\bigcup(\mathcal R\cup\mathcal X_r)$ with
$\htt(t)\geq \bar m$.  Then
$\eta^\frown\langle t\rangle\in T^{\uparrow<\omega}$ and
\[
  r^+=(\eta^\frown\langle t\rangle,h_r,\mathcal X_r)
\]
is a condition extending $r$.  Therefore $r^+\Vdash \ell\in\dot D$.
Hence $\eta^\frown\langle t\rangle\in B_\ell\subseteq W^\ell$.  Since also
$t\notin\bigcup\mathcal R$ and $\htt(t)\geq i$, we conclude again
that $\ell\in D_{\eta,i}$.

For (2), fix $\eta$ with $0<\rho_\ell(\eta)<\omega_1$, and put
$\alpha=\rho_\ell(\eta)$.  Let
\[
  L=
  \{t\in T\setminus\bigcup\mathcal R:
      \eta^\frown\langle t\rangle\in T^{\uparrow<\omega}
      \text{ and }
      \rho_\ell(\eta^\frown\langle t\rangle)<\alpha\}.
\]
As in the proof of (1), $|L|=\omega$.
For $\beta\in I$, let
\[
  U_\beta=\bigcup\mathcal S_\beta.
\]

Call $\beta$ bad if $L\setminus U_\beta$ is finite.  We show that at most one
$\beta$ is bad.  Suppose, towards a contradiction, that $\beta\neq\gamma$ are
both bad.  Then $L\setminus U_\beta$ and $L\setminus U_\gamma$ are finite, so
$L\cap U_\beta\cap U_\gamma$ is cofinite in $L$, and therefore infinite.

On the other hand, $U_\beta\cap U_\gamma$ is finite, as $\mathcal S_\beta$ and
$\mathcal S_\gamma$ are disjoint finite subfamilies of the almost disjoint
family $\mathcal X$.  This contradicts the infinitude of
$L\cap U_\beta\cap U_\gamma$.

Thus at most one $\beta\in I$ is bad.  Equivalently, for all but at most one
$\beta\in I$,
\[
  \{t\in T\setminus\bigcup(\mathcal R\cup\mathcal S_\beta):
        \eta^\frown\langle t\rangle\in T^{\uparrow<\omega}
        \text{ and }
        \rho_\ell(\eta^\frown\langle t\rangle)<\rho_\ell(\eta)
  \}
\]
is infinite.  This proves (2).
\end{proof}

\subsection{Preservation of intertwined families}
\label{subsec:preservation-intertwined}

\begin{lemma}\label{lem:QT-preservation-reduction}
Let $T$ be an $\omega$-tree, let $\mathcal X$ be an almost disjoint family of
subtrees of $T$, and assume that $\mathcal I_T(\mathcal X)$ is proper.  Let
$\langle\mathcal B,\mathcal C\rangle=\langle\{B_\alpha:\alpha<\omega_1\},
\{C_\alpha:\alpha<\omega_1\}\rangle$
be an intertwined family.  Let $\dot D$ be a $Q_T(\mathcal X)$-name for a
subset of $\omega$, and let $p\in Q_T(\mathcal X)$ be such that
\[
  p\Vdash \dot D\cap B_\beta\text{ is finite for uncountably many }\beta<\omega_1.
\]
Then there are $k<\omega$, an uncountable $I\subseteq\omega_1$, a condition
$q=(\sigma,h,\mathcal R)\leq p$, and conditions
\[
  q_\beta=(\sigma,h_\beta,\mathcal R\cup\mathcal S_\beta)
  \qquad(\beta\in I)
\]
such that:
\begin{enumerate}[label=(\alph*)]
    \item $q_\beta\leq p$ and
    $q_\beta\Vdash \dot D\cap B_\beta\subseteq k$ for every $\beta\in I$;
    \item the finite sets $\mathcal S_\beta$ are pairwise disjoint and
    disjoint from $\mathcal R$;
    \item for every $\tau\in T^{\uparrow<\omega}$ to which $q$ is extensible,
    the set
    \[
      \{\beta\in I:q_\beta\text{ is extensible to }\tau\}
    \]
    is uncountable;
    \item for every $\beta\in I$, every $\ell<\omega$, and every
    $\eta\in T^{\uparrow<\omega}$ with $0<\rho_\ell(\eta)<\omega_1$, the set
    \[
      \{t\in T\setminus\bigcup(\mathcal R\cup\mathcal S_\beta):
        \eta^\frown\langle t\rangle\in T^{\uparrow<\omega}
        \text{ and }
        \rho_\ell(\eta^\frown\langle t\rangle)<\rho_\ell(\eta)
      \}
    \]
    is infinite, where the ranks are computed from $\mathcal R$ and $\dot D$.
\end{enumerate}
\end{lemma}

\begin{proof}
For uncountably many $\beta<\omega_1$, choose $q_\beta\leq p$ and
$k_\beta<\omega$ such that $q_\beta\Vdash \dot D\cap B_\beta\subseteq k_\beta$.

Thinning the index set to an uncountable set $I_1$, we may assume that $k_\beta=k$ for all $\beta$ under
consideration and that all $q_\beta$ have the same stem $\sigma$.  By the
$\Delta$-system lemma, thin further to an uncountable set $I_2$ so that
\[
  q_\beta=(\sigma,h_\beta,\mathcal R\cup\mathcal S_\beta),
\]
where the finite sets $\mathcal S_\beta$ are pairwise disjoint and disjoint
from $\mathcal R$.

Write $p=(\bar\sigma,h_p,\mathcal X_p)$.  Since $q_\beta\leq p$ for every
$\beta$, we have $\mathcal X_p\subseteq\mathcal R$, and
$(\sigma,h_p,\mathcal R)\leq p$.  Applying
Lemma~\ref{lem:QT-preparation}, choose $h\geq h_p$ such that
\[
  q=(\sigma,h,\mathcal R)
\]
extends $p$ and, for every $\tau$ to which $q$ is extensible,
\[
  I_3=\{\beta:q_\beta\text{ is extensible to }\tau\}
\]
is uncountable.

For each fixed pair $(\eta,\ell)$ with $0<\rho_\ell(\eta)<\omega_1$,
Lemma~\ref{lem:QT-rank-facts} says that at most one $\beta$ fails clause (d).
Since $T^{\uparrow<\omega}\times\omega$ is countable, we may discard countably
many indices and obtain an uncountable set $I\subseteq\omega_1$ for which
clause (d) holds for every $\beta\in I$.  Clause (c) is preserved after this
countable thinning.
\end{proof}

\begin{lemma}\label{lem:QT-case-one-impossible}
Let $T$, $\mathcal X$, $\langle\mathcal B,\mathcal C\rangle$, $\dot D$,
$k$, $I$, $q=(\sigma,h,\mathcal R)$, and
$q_\beta=(\sigma,h_\beta,\mathcal R\cup\mathcal S_\beta)$ be as in
Lemma~\ref{lem:QT-preservation-reduction}.  Define the ranks and the sets
$D_{\eta,i}$ using $\mathcal R$ and $\dot D$.  Then there is no
$\tau\in T^{\uparrow<\omega}$ to which $q$ is extensible such that
\[
  \{\alpha<\omega_1:
    \forall i<\omega\ |D_{\tau,i}\cap C_\alpha|=\omega\}
\]
is uncountable.
\end{lemma}

\begin{proof}
Suppose that such a $\tau$ exists, and let $\Theta=\{\alpha<\omega_1:
\forall i<\omega\ |D_{\tau,i}\cap C_\alpha|=\omega\}$.
Call a set $E\subseteq\omega$ $\mathcal C$-large if
\[
  \{\alpha<\omega_1: |E\cap C_\alpha|=\omega\}
\]
is uncountable.
For $u\in T$ and $i<\omega$, define
\[
  D^u_i=
  \{\ell<\omega:
    \exists t\in T\setminus\bigcup\mathcal R\ 
    (
      \tau^\frown\langle t\rangle\in W^\ell,\ u\preceq t,
      \text{ and }\htt(t)\geq i
    )\}.
\]

Let $S$ be the set of all nodes $u\in T$ such that $D^u_i$ is
$\mathcal C$-large for every $i<\omega$.  The root of $T$ belongs to $S$.
We claim that every $u\in S$ has an immediate successor in $S$.  Otherwise,
for each immediate successor $v$ of $u$ choose $i_v<\omega$ such that
$D^v_{i_v}$ is not $\mathcal C$-large.  Taking $i$ larger than $\htt(u)$ and
all the $i_v$'s, every witness for membership in $D^u_i$ lies above one of
these immediate successors.  Thus $D^u_i$ is contained in a finite union of
sets that are not $\mathcal C$-large, contradicting $u\in S$.  Hence $S$
contains an infinite branch.  Let
$b=\{b_n:n<\omega\}$ be such a branch, where $b_n$ is the unique node of $b$ at level $n$.

For $n<\omega$ define
\[
  E^0_n=
  \{\ell<\omega:\exists j\geq n\ b_j \in T\setminus\bigcup\mathcal R\text{ and }
    \tau^\frown\langle b_j\rangle\in W^\ell\}
\]
and
\[
  E^1_n=
  \{\ell<\omega: \exists t\in T\setminus\bigcup\mathcal R\ 
    \tau^\frown\langle t\rangle\in W^\ell,\ b_n\preceq t
    \text{ and } t\notin b\}.
\]

For every $n<\omega$ and $i<\omega$, notice that
$D^{b_n}_i\subseteq E^0_n\cup E^1_n$. Since $b_n\in S$, at least one of the
sets $E^0_n,E^1_n$ is $\mathcal C$-large.
Since both sequences $(E^0_n:n<\omega)$ and $(E^1_n:n<\omega)$ are decreasing,
there is $\varepsilon\in\{0,1\}$ such that, for every $n<\omega$,
$E^\varepsilon_n$ is $\mathcal C$-large.

Let
\[
  I_\tau=\{\beta\in I:q_\beta\text{ is extensible to }\tau\}.
\]
This set is uncountable by Lemma~\ref{lem:QT-preservation-reduction}.
Since $\mathcal X$ is almost disjoint and $b$ is infinite, at most one member of
$\mathcal X$ contains $b$.
Since the sets $\mathcal S_\beta$ are pairwise disjoint, at most one
$\beta\in I_\tau$ has some member of $\mathcal S_\beta$ containing $b$.
Shrinking $I_\tau$ by removing this possible exception, we may assume that no
member of any $\mathcal S_\beta$, $\beta\in I_\tau$, contains $b$.
We claim that, for every $\beta\in I_\tau$, there is $n_\beta<\omega$ such that
\[
  \{t\in T:b_{n_\beta}\preceq t\}\cap\bigcup\mathcal S_\beta=\emptyset.
\]
Indeed, if this failed for some $\beta\in I_\tau$, then for every
$n<\omega$ there would be $t_n\in\bigcup\mathcal S_\beta$ with
$b_n\preceq t_n$.  Since $\mathcal S_\beta$ is finite, some
$A\in\mathcal S_\beta$ contains $t_n$ for infinitely many $n$.  As $A$ is
downward closed, $b_n\in A$ for infinitely many $n$.  Thus, for every
$m<\omega$, we can choose $n\geq m$ with $b_n\in A$; since
$b_m\preceq b_n$ and $A$ is downward closed, $b_m\in A$.  Hence
$b\subseteq A$, contradicting the choice of $I_\tau$.
Thinning $I_\tau$, choose an uncountable $I'\subseteq I_\tau$, a level
$n_0<\omega$, and $j_0<\omega$ such that, for every $\beta\in I'$,
\[
  \{t\in T:b_{n_0}\preceq t\}\cap\bigcup\mathcal S_\beta=\emptyset
  \qquad\text{and}\qquad
  h_\beta(\tau)\leq j_0.
\]

Choose $n\geq\max\{n_0,j_0\}$.
By the choice of $n$, whenever $\ell\in E^\varepsilon_n$ and $t$ witnesses
this membership, the node $t$ lies above $b_{n_0}$, avoids
$\bigcup(\mathcal R\cup\mathcal S_\beta)$, and has height at least
$h_\beta(\tau)$, for every $\beta\in I'$.  Hence $q_\beta$ is extensible to
$\tau^\frown\langle t\rangle$ for every $\beta\in I'$.

By intertwinedness, since $E^\varepsilon_n$ is $\mathcal C$-large,
$E^\varepsilon_n$ has infinite intersection with
$B_\beta$ for all but countably many $\beta<\omega_1$.  Choose
$\beta\in I'$ such that $|E^\varepsilon_n\cap B_\beta|=\omega$, and pick
$\ell\in E^\varepsilon_n\cap B_\beta$ with $\ell\geq k$.  Let $t$ witness
that $\ell\in E^\varepsilon_n$, and put
$\eta_0=\tau^\frown\langle t\rangle$.  Then $q_\beta$ is extensible to
$\eta_0$ and $\rho_\ell(\eta_0)<\omega_1$.

Starting from $\eta_0$, we construct a finite descending sequence of
$\ell$-ranks, keeping each stem extensible from $q_\beta$, as follows.  Suppose that
$\eta_j$ has been constructed, $q_\beta$ is extensible to $\eta_j$, and
$\rho_\ell(\eta_j)>0$.  By clause (d) of
Lemma~\ref{lem:QT-preservation-reduction}, there are rank-lowering nodes
outside $\bigcup(\mathcal R\cup\mathcal S_\beta)$ of arbitrarily large height.
Choose such a node $s$ with $\htt(s)\geq h_\beta(\eta_j)$.
Then $q_\beta$ is extensible to
$\eta_{j+1}=\eta_j^\frown\langle s\rangle$, and $\rho_\ell(\eta_{j+1})<\rho_\ell(\eta_j)$.

Since there is no infinite strictly decreasing sequence of ordinals, this
process stops after finitely many steps.  Thus we obtain a stem $\eta$ such
that $q_\beta$ is extensible to $\eta$ and $\rho_\ell(\eta)=0$.

By the definition of rank zero, there is a condition $r$ with stem $\eta$ such
that $r\Vdash \ell\in\dot D$.
Since $q_\beta$ is extensible to $\eta$, the condition
\[
  q_\beta^\eta=(\eta,h_\beta,\mathcal R\cup\mathcal S_\beta)
\]
is an extension of $q_\beta$.  The conditions $q_\beta^\eta$ and $r$ have the
same stem, so by Lemma~\ref{lem:QT-sigma-centered} they are compatible.  Let
$s^*$ be a common extension.  Then $s^*\leq r$, so $s^*\Vdash \check \ell\in\dot D$.
On the other hand, $s^*\leq q_\beta$ and
\[
  q_\beta\Vdash \dot D\cap \check B_\beta\subseteq \check k.
\]
Since $\ell\in B_\beta$ and $\ell\geq k$, this implies $s^*\Vdash \check \ell\notin\dot D$, a contradiction.
\end{proof}

\begin{lemma}\label{lem:QT-case-two-forces-finite}
Let $T$, $\mathcal X$, $\langle\mathcal B,\mathcal C\rangle$, $\dot D$,
$k$, $I$, $q=(\sigma,h,\mathcal R)$, and
$q_\beta=(\sigma,h_\beta,\mathcal R\cup\mathcal S_\beta)$ be as in
Lemma~\ref{lem:QT-preservation-reduction}.  Define the ranks and the sets
$D_{\eta,i}$ using $\mathcal R$ and $\dot D$.  Suppose that, for every
$\eta\in T^{\uparrow<\omega}$ to which $q$ is extensible, the set
\[
  \Theta_\eta=
  \{\alpha<\omega_1:
    \forall i<\omega\ |D_{\eta,i}\cap C_\alpha|=\omega\}
\]
is countable.  Then
\[
  q\Vdash \{\alpha<\omega_1:|\dot D\cap \check C_\alpha|=\omega\}\text{ is countable}.
\]
\end{lemma}

\begin{proof}
  Let $\Theta=\bigcup\{\Theta_\eta:\eta\in T^{\uparrow<\omega}\text{ and }q\text{ is extensible to }\eta\}$.
  Since $T^{\uparrow<\omega}$ is countable and each $\Theta_\eta$ is countable,
  $\Theta$ is countable.
  We shall prove that for every $\alpha\notin\Theta$, $q\Vdash \dot D\cap \check C_\alpha\text{ is finite}$.

  It is enough to show that, for every $\alpha\notin\Theta$ and every
  $r\leq q$, there is $r^*\leq r$ and a finite $E\subseteq C_\alpha$ such that $r^*\Vdash \dot D\cap \check C_\alpha\subseteq \check E$.

  Fix $\alpha\notin\Theta$ and $r=(\eta,h_r,\mathcal X_r)\leq q$.  In particular,
  $q$ is extensible to $\eta$.

  Since $\alpha\notin\Theta_\eta$, choose $i_\eta<\omega$ such that
  \[
    E=D_{\eta,i_\eta}\cap C_\alpha
  \]
  is finite.  If $\ell\in C_\alpha\setminus E$, then
  $\rho_\ell(\eta)=\infty$.  Indeed, if $\rho_\ell(\eta)<\omega_1$, then
  $\eta\in W^\ell$, and Lemma~\ref{lem:QT-rank-facts} gives
  $\ell\in D_{\eta,i_\eta}$, contradicting $\ell\notin E$.
  Thus:
  \begin{equation}\label{eq:QT-case-two-initial-ranks}\tag{$\star$}
    \forall \ell\in C_\alpha\setminus E\, \rho_\ell(\eta)=\infty.
  \end{equation}

  For each $n<\omega$, put $L_n=\{\nu\in T^{\uparrow<\omega}:|\nu|=n\}$ and $L_{<n}=\bigcup_{m<n}L_m$.
  We define a sequence of functions $(g_n:n<\omega)$ such that, for every $n<\omega$:
  \begin{enumerate}[label=(\arabic*)]
    \item $g_n:T^{\uparrow<\omega}\to\omega$; \item $g_n|(T^{\uparrow<\omega}\setminus L_{<n})=h_r|(T^{\uparrow<\omega}\setminus L_{<n})$;
    \item for all $m<n$, $g_n|L_{<m}=g_m|L_{<m}$ and $g_m\leq g_n$ pointwise;
    \item $g_n\geq h_r$ pointwise and $(\eta, g_n,\mathcal X_r)\in Q_T(\mathcal X)$;
    \item for all $\nu \in L_n$, if $(\eta,g_n,\mathcal X_r)$ is not extensible to $\nu$, then $g_{n+1}(\nu)=h_r(\nu)$;
    \item for all $\nu\in L_n$, if $(\eta,g_n,\mathcal X_r)$ is
      extensible to $\nu$, then
      \begin{enumerate}[label=(\roman*)]
        \item $\forall \ell\in C_\alpha\setminus E\, \rho_\ell(\nu)=\infty$; and
        \item for every $t\in T$ such that $\nu^\frown\langle t\rangle\in T^{\uparrow<\omega}$, $t\notin\bigcup\mathcal X_r$, and $\htt(t)\geq g_{n+1}(\nu)$, we have $\rho_\ell(\nu^\frown\langle t\rangle)=\infty$ for every $\ell\in C_\alpha\setminus E$.
      \end{enumerate}
  \end{enumerate}

  We proceed by recursion on $n$.
Let $g_0=h_r$.  Then (1)--(4) hold, and there is nothing to verify for
(5) and (6).

Assume that $(g_m:m\leq n)$ has been constructed.  We define $g_{n+1}$ as
follows.  Put $g_{n+1}=g_n$ on $L_{<n}$ and put
$g_{n+1}=h_r$ on $T^{\uparrow<\omega}\setminus L_{<n+1}$.

Now fix $\nu\in L_n$.  If $(\eta,g_n,\mathcal X_r)$ is not extensible to
$\nu$, let $g_{n+1}(\nu)=h_r(\nu)$.

Suppose, then, that $(\eta,g_n,\mathcal X_r)$ is extensible to $\nu$.
We first verify (6)(i).  If $\nu=\eta$, this follows by Equation~\eqref{eq:QT-case-two-initial-ranks}.  Otherwise,
write $\nu=\mu^\frown\langle t\rangle$, where $\mu\in L_{n-1}$.  Since
$(\eta,g_n,\mathcal X_r)$ is extensible to $\nu$, we have
\[
  t\notin\bigcup\mathcal X_r
  \qquad\text{and}\qquad
  \htt(t)\geq g_n(\mu).
\]
Moreover, $(\eta,g_{n-1},\mathcal X_r)$ is extensible to $\mu$, because
$g_{n-1}$ and $g_n$ agree on all proper initial segments of $\mu$ relevant
to extensibility.  By the induction hypothesis applied to $\mu$, specifically by (6)(ii), we get $\rho_\ell(\nu)=\infty$ for every $\ell\in C_\alpha\setminus E$.

Since $g_n\geq h_r$ and $r\leq q$, every such $\nu$ is also an extension to
which $q$ is extensible.  Therefore $\alpha\notin\Theta_\nu$, and we may
choose $i(\nu)<\omega$ such that $D_{\nu,i(\nu)}\cap C_\alpha$ is finite.
Let
\[
  F_\nu=(D_{\nu,i(\nu)}\cap C_\alpha)\setminus E.
\]
For each $\ell\in F_\nu$, (6)(i) gives $\rho_\ell(\nu)=\infty$.
We claim that the set
\[
  L_{\nu,\ell}=
  \{t\in T\setminus\bigcup\mathcal R:
    \nu^\frown\langle t\rangle\in T^{\uparrow<\omega}
    \text{ and }
    \rho_\ell(\nu^\frown\langle t\rangle)<\omega_1\}
\]
is finite.
Indeed, if $L_{\nu,\ell}$ is infinite, then the set $\{\rho_\ell(\nu^\frown\langle t\rangle):t\in L_{\nu,\ell}\}$
is a countable set of countable ordinals.
Hence, it is bounded by some $\gamma<\omega_1$.  Thus, for every $t\in L_{\nu,\ell}$, $\nu^\frown\langle t\rangle\in\bigcup_{\xi<\gamma}W^\ell_\xi$.
Since $L_{\nu,\ell}\subseteq T\setminus\bigcup\mathcal R$ is infinite, the definition of $W^\ell_\gamma$ gives $\nu\in W^\ell_\gamma$, contradicting $\rho_\ell(\nu)=\infty$.

Since $F_\nu$ is finite, choose $j(\nu)\geq i(\nu)$ such that no node in
$\bigcup_{\ell\in F_\nu}L_{\nu,\ell}$ has height at least $j(\nu)$.  Set
\[
  g_{n+1}(\nu)=\max\{h_r(\nu),j(\nu)\}.
\]

We now verify (6)(ii).
Let $t\in T$ be such that $\nu^\frown\langle t\rangle\in T^{\uparrow<\omega}$, $t\notin\bigcup\mathcal X_r$, and $\htt(t)\geq g_{n+1}(\nu)$.

Fix $\ell\in C_\alpha\setminus E$.  Since $r\leq q$, we have
$\mathcal R\subseteq\mathcal X_r$, and hence $t\notin\bigcup\mathcal R$.
If $\ell\in F_\nu$, then $t\notin L_{\nu,\ell}$ by the choice of $j(\nu)$,
so $\rho_\ell(\nu^\frown\langle t\rangle)=\infty$.  If
$\ell\notin F_\nu$, then $\ell\notin D_{\nu,i(\nu)}$.  Since
$\htt(t)\geq j(\nu)\geq i(\nu)$ and
$t\notin\bigcup\mathcal R$, the definition of $D_{\nu,i(\nu)}$ again gives $\rho_\ell(\nu^\frown\langle t\rangle)=\infty$.

This completes the definition of $g_{n+1}$.  Conditions (1)--(3) are
immediate from the construction.  Condition (4) holds because the only values
of $g_{n+1}$ on proper initial segments of $\eta$ are the corresponding
values of $h_r$.  Condition (5) was arranged by definition, and condition
(6) was verified above.

This completes the recursive construction.
Define $h^*:T^{\uparrow<\omega}\to\omega$ by $h^*(\nu)=g_{n+1}(\nu)$ whenever $\nu\in L_n$.
Then $h^*\geq h_r$.  Moreover, if $\nu$ is not an extension of $\eta$, then
$(\eta,g_n,\mathcal X_r)$ is not extensible to $\nu$, where $n=|\nu|$, and
therefore $h^*(\nu)=g_{n+1}(\nu)=h_r(\nu)$.  In particular, $h^*=h_r$ on all
proper initial segments of $\eta$.

Let $r^*=(\eta,h^*,\mathcal X_r)$.  Since $h^*=h_r$ on all proper
initial segments of $\eta$ and $h^*\geq h_r$, the triple $r^*$ is a
condition and $r^*\leq r$.

We claim that $r^*\Vdash \dot D\cap \check C_\alpha\subseteq \check E$.
Suppose not.  Then there are $s\leq r^*$ and $\ell\in C_\alpha\setminus E$ such that $s\Vdash \check \ell\in\dot D$.
Let $\nu=\sigma_s$, and let $n=|\nu|$.  Since $s\leq r^*$, the condition
$r^*$ is extensible to $\nu$.

We claim that $(\eta,g_n,\mathcal X_r)$ is extensible to $\nu$.  Indeed, for
every proper initial segment $\tau$ of $\nu$, the definition of $h^*$ and
condition (3) of the construction give $h^*(\tau)=g_n(\tau)$.
The side condition is the same, namely $\mathcal X_r$.  Hence the
extensibility of $r^*$ to $\nu$ is exactly the extensibility of
$(\eta,g_n,\mathcal X_r)$ to $\nu$.

By (6)(i), applied to this $\nu\in L_n$, we have $\rho_\ell(\nu)=\infty$.
On the other hand, $s$ has stem $\nu$ and $s \Vdash \check \ell\in\dot D$.  Hence,
$\nu\in B_\ell$, and therefore $\rho_\ell(\nu)=0$, a contradiction.

Thus, $r^*\Vdash \dot D\cap \check C_\alpha\subseteq \check E$, completing the proof.
\end{proof}

\begin{lemma}\label{lem:tree-preserves-intertwined}
Let $T$ be an $\omega$-tree, and let $\mathcal X$ be an
almost disjoint family of subtrees of $T$.
Assume that $\mathcal I_T(\mathcal X)$ is proper.
Then $Q_T(\mathcal X)$ preserves intertwined families.
That is, if $\langle\mathcal B,\mathcal C\rangle$
is intertwined in the ground model, then
\[
  Q_T(\mathcal X)\Vdash
  \langle\check{\mathcal B},\check{\mathcal C}\rangle\text{ is intertwined}.
\]
\end{lemma}

\begin{proof}
The almost disjointness of $\mathcal B\cup\mathcal C$ is absolute.
We verify the second clause of intertwinedness by contraposition.
Enumerate $\mathcal B$ and $\mathcal C$ as $\{B_\alpha:\alpha<\omega_1\}$ and $\{C_\alpha:\alpha<\omega_1\}$, respectively.

Let $\dot D$ be a $Q_T(\mathcal X)$-name for a subset of $\omega$, and let
$p\in Q_T(\mathcal X)$ be such that
\[
  p \Vdash \dot D\cap \check B_\beta\text{ is finite for uncountably many }\beta<\omega_1.
\]
Apply Lemma~\ref{lem:QT-preservation-reduction}, and obtain
$k<\omega$, an uncountable $I\subseteq\omega_1$, a condition
$q=(\sigma,h,\mathcal R)\leq p$, and conditions
\[
  q_\beta=(\sigma,h_\beta,\mathcal R\cup\mathcal S_\beta)
  \qquad(\beta\in I)
\]
with the conclusions of that lemma.

Let the ranks $\rho_\ell$ and the sets $D_{\eta,i}$ be those defined in
Definition~\ref{def:QT-ranks}, with this $\mathcal R$ and this name $\dot D$.
For $\eta\in T^{\uparrow<\omega}$ to which $q$ is extensible, put $\Theta_\eta=
  \{\alpha<\omega_1:
    \forall i<\omega\ |D_{\eta,i}\cap C_\alpha|=\omega\}$.
By Lemma~\ref{lem:QT-case-one-impossible}, each $\Theta_\eta$ is countable.
Therefore Lemma~\ref{lem:QT-case-two-forces-finite} gives
\[
  q\Vdash_{Q_T(\mathcal X)}
  \{\alpha<\omega_1:|\dot D\cap \check C_\alpha|=\omega\}\text{ is countable}.
\]
This proves the contrapositive of the defining implication for intertwinedness.
\end{proof}

\begin{lemma}\label{lem:fs-iteration-preserves-intertwined}
Let $\langle\mathbb P_\xi,\dot{\mathbb Q}_\xi:\xi<\delta\rangle$
be a finite-support iteration of forcing notions with the countable chain condition.
Let $\langle\mathcal B,\mathcal C\rangle$ be an intertwined family in the ground
model.  Assume that, for every $\xi<\delta$,
\[
  \mathbb P_\xi\Vdash
  \dot{\mathbb Q}_\xi\text{ preserves }
  \langle\check{\mathcal B},\check{\mathcal C}\rangle
  \text{ as an intertwined family}.
\]
Then
\[
  \mathbb P_\delta\Vdash
  \langle\check{\mathcal B},\check{\mathcal C}\rangle
  \text{ is intertwined}.
\]
\end{lemma}

\begin{proof}
We argue by induction on $\delta$.  Successor stages follow from the
hypothesis.  Let $\delta$ be a limit ordinal.  The almost disjointness of
$\mathcal B\cup\mathcal C$ is absolute, so it remains to preserve the second
clause of intertwinedness.
Write $\mathcal B=\{B_\alpha:\alpha<\omega_1\}$ and $\mathcal C=\{C_\alpha:\alpha<\omega_1\}$.

Let $\dot D$ be a nice $\mathbb P_\delta$-name for a subset of $\omega$, and let
$p\in\mathbb P_\delta$ be such that
\[
  p\Vdash_{\mathbb P_\delta}
  \dot D\cap \check B_\beta\text{ is finite for uncountably many }\beta<\omega_1.
\]
For uncountably many such $\beta$, choose
$p_\beta\leq p$ and $k_\beta<\omega$ such that
\[
  p_\beta\Vdash_{\mathbb P_\delta}
  \dot D\cap \check B_\beta\subseteq\check k_\beta.
\]
Thinning the index set, assume $k_\beta=k$ for all $\beta$ under
consideration.

If $\operatorname{cf}(\delta)>\omega$, then, since $\mathbb P_\delta$ is ccc
and $\dot D$ is a nice name for a subset of $\omega$, there is $\xi<\delta$ such
that $\dot D$ is a $\mathbb P_\xi$-name and $p\in\mathbb P_\xi$.
Moreover,
\[
  p\Vdash_{\mathbb P_\xi}
  \dot D\cap B_\beta\text{ is finite for uncountably many }\beta<\omega_1.
\]
Indeed, otherwise some condition in $\mathbb P_\xi$ below $p$ would force the
negation, and the same condition, regarded as a condition in $\mathbb P_\delta$,
would contradict the choice of $p$.

By the induction hypothesis applied to $\mathbb P_\xi$, there is
$q\leq p$ in $\mathbb P_\xi$ such that
\[
  q\Vdash_{\mathbb P_\xi}
  \dot D\cap C_\alpha\text{ is finite for all but countably many }
  \alpha<\omega_1.
\]
Since $\mathbb P_\xi$ is a complete suborder of $\mathbb P_\delta$ and
$\dot D$ is a $\mathbb P_\xi$-name, the same $q$, regarded as a condition in
$\mathbb P_\delta$, forces the same conclusion.

It remains to consider the case $\operatorname{cf}(\delta)=\omega$.  Fix an
increasing sequence $(\delta_n:n<\omega)$ cofinal in $\delta$.  Since each
$p_\beta$ has finite support, there are $n<\omega$ and an uncountable set
$I\subseteq\omega_1$ such that $p\in\mathbb P_{\delta_n}$ and
$p_\beta\in\mathbb P_{\delta_n}$ for every $\beta\in I$.

We claim there is a condition $q\leq p$ in
$\mathbb P_{\delta_n}$ such that
\[
  q\Vdash_{\mathbb P_{\delta_n}}
  \{\beta\in I:p_\beta\in\dot G_{\delta_n}\}\text{ is uncountable}.
\]
Indeed, suppose otherwise.  Then $p$ would force this set to be countable, so
by the countable chain condition, there exists a countable set $I'\subseteq I$ such that
\[
  p\Vdash_{\mathbb P_{\delta_n}}
  \{\beta\in I:p_\beta\in\dot G_{\delta_n}\}\subseteq \check I'.
\]
But then, for each $\beta\in I$, $p_\beta \Vdash \check \beta\in\check I'$, so $I\subseteq I'$ in the ground model, a contradiction.

Now working in a countable transitive model $M$, let $G_{\delta_n}$ be a $\mathbb P_{\delta_n}$-generic filter over $M$ with $q \in G_{\delta_n}$.
Since $q\leq p$, the generic filter also contains $p$, so $J=\{\beta\in I:p_\beta\in G_{\delta_n}\}$ is uncountable.
Define
\[
  D^*=\{\ell<\omega:
    \exists r\in\mathbb P_\delta\,
    (r\restriction\delta_n\in G_{\delta_n}
    \text{ and }
    r\Vdash_{\mathbb P_\delta}\check\ell\in\dot D)\}.
\]
For every $\beta\in J$, we have $D^*\cap B_\beta\subseteq k$.
Otherwise, if $\ell\in D^*\cap B_\beta\setminus k$, let $r \in \mathbb P_\delta$ be such that $r\restriction\delta_n\in G_{\delta_n}$ and $r\Vdash_{\mathbb P_\delta}\check\ell\in\dot D$.
Since $p_\beta\in G_{\delta_n}$ and $r\restriction\delta_n\in G_{\delta_n}$, there is a condition in $G_{\delta_n}$ extending both $p_\beta$ and $r\restriction\delta_n$.
Combining it with the tail of $r$, we get a common extension $s\leq p_\beta,r$ in $\mathbb P_\delta$.
Then $s\Vdash \check \ell\in\dot D$.
On the other hand, $s\leq p_\beta$ and $p_\beta\Vdash \dot D\cap\check B_\beta\subseteq\check k$; since $\ell\in B_\beta\setminus k$, this gives $s\Vdash \check \ell\notin\dot D$, a contradiction.

By the induction hypothesis, the pair
$\langle\mathcal B,\mathcal C\rangle$ is intertwined in the
$\mathbb P_{\delta_n}$-extension.  Therefore, since $J$ is uncountable and
$D^*\cap B_\beta\subseteq k$ for every $\beta\in J$, the set $\Theta=\{\alpha<\omega_1:|D^*\cap C_\alpha|=\omega\}$
is countable in $V[G_{\delta_n}]$.

Let $H$ be a $\mathbb P_\delta$-generic
filter extending $G_{\delta_n}$.
If $\ell\in\dot D_H$, then some $r\in H$ forces
$\check\ell\in\dot D$.  Since $r\restriction\delta_n\in G_{\delta_n}$, the
definition of $D^*$ gives $\ell\in D^*$.  Hence $\dot D_H\subseteq D^*$.
Consequently, for every $\alpha\notin\Theta$,
\[
  |\dot D_H\cap C_\alpha|
  \leq |D^*\cap C_\alpha|
  <\omega.
\]
Thus, in every $\mathbb P_\delta$-generic extension containing $q$,
$\dot D\cap C_\alpha$ is finite for all but countably many $\alpha<\omega_1$.
Equivalently,
\[
  q\Vdash_{\mathbb P_\delta}
  \dot D\cap \check C_\alpha\text{ is finite for all but countably many }
  \alpha<\omega_1.
\]
Since $q\leq p$, this proves the limit step and completes the induction.
\end{proof}

\subsection{The iteration and final model}

In the definition below, it would make sense for us to restrict $S$ to a subset of $\kappa\times\omega$ for some cardinal $\kappa$, but this restriction is not needed.
The reader may keep in mind that this is the intended context for the definition.

\begin{definition}\label{def:Col}
  Let $S$ be a relation.
  The \emph{collection determined by $S$} is $\Col(S)=\{\{n\in\omega: (a,n)\in S\}: a \in \dom S\}$.
\end{definition}

We shall use the following bounded form of the maximal principle.  Let
$\kappa$ be a regular cardinal and let $\mathbb P\in H(\kappa)$ be a forcing
preorder, where $H(\kappa)$ denotes the set of all sets whose transitive closures
have cardinality less than $\kappa$.  If $p\in\mathbb P$ and
\[
  p\Vdash_{\mathbb P}\exists x\bigl(x\in H(\check\kappa)\wedge\varphi(x)\bigr),
\]
where $\varphi$ is a formula with $x$ as its only free variable, possibly with
$\mathbb P$-names as parameters, then there is a $\mathbb P$-name
$\dot x\in H(\kappa)$ such that $p\Vdash_{\mathbb P}\dot x\in H(\check\kappa)\wedge\varphi(\dot x)$.

\begin{thm}\label{thm:consistent-ap-below-at}
Assume GCH and let $\lambda>\omega_1$ be regular.  There is a ccc forcing
extension in which
\[
  \mathfrak{ap}=\omega_1<
  \mathfrak{at}=\mathfrak q_2=\mathfrak q_{2\frac{1}{2}}=\mathfrak c=\lambda.
\]
In particular,
\[
  \operatorname{Con}(\mathrm{ZFC})
  \Rightarrow
  \operatorname{Con}(\mathrm{ZFC}+\mathfrak{ap}<\mathfrak{at}).
\]
\end{thm}

\begin{proof}
Start with a model of GCH.  By \cite[Lemma~2.5]{brendle1999dow}, there exists an
intertwined family
\[
  \langle\mathcal B,\mathcal C\rangle
  =\langle\{B_\alpha:\alpha<\omega_1\},
           \{C_\alpha:\alpha<\omega_1\}\rangle.
\]

Fix a surjection $f:\lambda\to\lambda\times\lambda$ such that, for every
$\xi, \beta, \eta<\lambda$, if $f(\xi)=(\beta,\eta)$, then $\eta\leq\xi$.
Also, fix a well-ordering $\sqsubseteq$ of $H(\lambda^+)$.

We shall construct a finite-support iteration
\[
  \bigl\langle((\mathbb P_\xi, \leq_\xi, \mathbbm 1_\xi):\xi\leq\lambda),
  ((\dot{\mathbb Q}_\xi, \dot{\leq}_\xi, \dot{\mathbbm 1}_\xi):\xi<\lambda)\bigr\rangle
\]
along with
\[
  \Bigl(\bigl(\dot{\preceq}^\xi_\beta,\dot S^\xi_\beta,\kappa^\xi_\beta:\beta<\lambda\bigr):\xi<\lambda\Bigr)
\]
such that:

\begin{enumerate}[label=(\arabic*)]
  \item for every $\xi\leq\lambda$, $(\mathbb P_\xi, \leq_\xi, \mathbbm 1_\xi)$ is a forcing preorder with the countable chain condition and top element $\mathbbm 1_\xi=\emptyset$, and $\mathbb P_\xi$ is a set of functions whose domains are finite subsets of $\xi$;\label{item:iteration-ccc}
  \item $\mathbb P_0=\{\emptyset\}$ is trivially ordered;\label{item:iteration-trivial}
  \item for every $\xi<\lambda$, $(\mathbb P_\xi, \leq_\xi, \mathbbm 1_\xi)\in H(\lambda^+)$ and $(\dot{\mathbb Q}_\xi, \dot{\leq}_\xi, \dot{\mathbbm 1}_\xi)\in H(\lambda^+)$;\label{item:hereditarily-small}
  \item for every $\beta, \xi\leq \lambda$, if $\beta\leq \xi$ then $\mathbb P_\beta$ is a complete suborder of $\mathbb P_\xi$;
  \item for every $\xi<\lambda$, $\dot{\mathbb Q}_\xi$, $\dot{\leq}_\xi$,
  and $\dot{\mathbbm 1}_\xi$ are $\mathbb P_\xi$-names, and
  \[
    \mathbb P_\xi\Vdash
    (\dot{\mathbb Q}_\xi,\dot{\leq}_\xi,\dot{\mathbbm 1}_\xi)
    \text{ is a ccc forcing preorder};
  \]
  \item for $\xi<\lambda$, for every $p \in \mathbb P_\xi$ and $\alpha \in \dom p$, $p(\alpha)\in \dom \dot{\mathbb Q}_\alpha$ and $\mathbb P_\xi\Vdash p(\alpha)\in \dot{\mathbb Q}_\alpha$;
  \item \label{item:successor-step} for every $\xi<\lambda$,
  \[
    \mathbb P_{\xi+1}=\mathbb P_\xi\cup
    \{p\cup\{(\xi,\dot q)\}:p\in\mathbb P_\xi,\ 
    \dot q\in\dom\dot{\mathbb Q}_\xi,\text{ and }
    \mathbb P_\xi\Vdash \dot q\in\dot{\mathbb Q}_\xi\};
  \]
  \item \label{item:successor-step-2} for every $\xi<\lambda$, if $p,q\in\mathbb P_{\xi+1}$, then
  $p\leq_{\xi+1}q$ if and only if $\dom q\subseteq\dom p$,
  $p|\xi\leq_\xi q|\xi$, and if $\xi\in\dom q$, then
  \[
    p|\xi\Vdash_\xi p(\xi)\,\dot{\leq}_\xi\, q(\xi);
  \]
  \item if $\xi\leq \lambda$ is a limit ordinal, then
  $\mathbb P_\xi=\bigcup_{\beta<\xi}\mathbb P_\beta$, and
  $\leq_\xi=\bigcup_{\beta<\xi}\leq_\beta$; \label{item:limit-step}
  \item \label{item:listing} for every $\xi<\lambda$, $\bigl((\dot{\preceq}^\xi_\beta,\dot S^\xi_\beta,\kappa^\xi_\beta):
    \beta<\lambda\bigr)$
  is the $\sqsubseteq$-least sequence of length $\lambda$, with repetitions allowed, whose range is the
  set of all triples
  $(\dot{\preceq},\dot S,\kappa)$ such that $\kappa<\lambda$ is a cardinal,
  $\dot{\preceq}$ is a $\mathbb P_\xi$-nice name for a subset of
  $\check{\omega\times\omega}$, $\dot S$ is a $\mathbb P_\xi$-nice name for
  a subset of $\check{\kappa\times\omega}$, and
  \[
  \begin{aligned}
  \mathbb P_\xi\Vdash{}&(\omega, \dot{\preceq}) \text{ is an } \omega\text{-tree and}\\
  &\Col(\dot S)\text{ is an infinite almost disjoint family of subtrees of }(\omega,\dot{\preceq})
  \end{aligned}
  \]
  \item \label{item:forcing-definitions} for every $\xi,\beta,\eta<\lambda$,
  if $f(\xi)=(\beta,\eta)$, then $\mathbb P_\xi\Vdash
    (\dot{\mathbb Q}_\xi,\dot{\leq}_\xi,\dot{\mathbbm 1}_\xi)
    =
    Q_{(\omega,\dot{\preceq}^\eta_\beta)}
    \bigl(\Col(\dot S^\eta_\beta)\bigr)$.
\end{enumerate}

We now carry out the finite-support recursion.
Assume we have defined $\langle((\mathbb P_\beta, \leq_\beta, \mathbbm 1_\beta):\beta<\xi), ((\dot{\mathbb Q}_\beta, \dot{\leq}_\beta, \dot{\mathbbm 1}_\beta):\beta<\xi)\rangle$ and $\bigl((\dot{\preceq}^\beta_\alpha,\dot S^\beta_\alpha,\kappa^\beta_\alpha:\alpha<\lambda):\beta<\xi\bigr)$ for some $\xi\leq\lambda$.
We show how to define $(\mathbb P_\xi, \leq_\xi, \mathbbm 1_\xi)$, and, in case $\xi<\lambda$, how to define $(\dot{\mathbb Q}_\xi, \dot{\leq}_\xi, \dot{\mathbbm 1}_\xi)$ and $\bigl((\dot{\preceq}^\xi_\beta,\dot S^\xi_\beta,\kappa^\xi_\beta:\beta<\lambda)\bigr)$.

If $\xi=0$, then $\mathbb P_0=\{\emptyset\}$ is ordered trivially.
If $\xi\leq\lambda$ is a limit ordinal, we take unions as in item~\ref{item:limit-step}.
For the successor step $\xi=\zeta+1<\lambda$, we define
$\mathbb P_{\zeta+1}$ as in items~\ref{item:successor-step}
and~\ref{item:successor-step-2}.
This defines $(\mathbb P_\xi,\leq_\xi,\mathbbm 1_\xi)$; the corresponding iteration clauses follow from the induction hypothesis and from the usual finite-support iteration construction.

Now, assume $\xi<\lambda$. 
By item~\ref{item:hereditarily-small}, $|\mathbb P_\xi|\leq\lambda$.
Moreover, finite-support iterations of ccc posets are ccc, so
$\mathbb P_\xi$ is ccc.
There are at most $\bigl|[\mathbb P_\xi]^\omega\bigr|^{|\omega\times\omega|}
=|\mathbb P_\xi|^\omega\leq\lambda^\omega=\lambda$ nice names for subsets of $\check{\omega\times\omega}$.
Likewise, given $\kappa<\lambda$, there are at most
$|\mathbb P_\xi|^{\max\{\kappa,\omega\}}\leq \lambda$ nice names for subsets of $\check{\kappa\times \omega}$.
As there are at most $\lambda$ many cardinals below $\lambda$, the set of triples described in item~\ref{item:listing} has cardinality at most $\lambda$, and we may enumerate it as
$\bigl((\dot{\preceq}^\xi_\beta,\dot S^\xi_\beta,\kappa^\xi_\beta):\beta<\lambda\bigr)$, allowing repetitions.
Choose the $\sqsubseteq$-least such enumeration.

Finally, to define $\dot{\mathbb Q}_\xi$, $\dot{\leq}_\xi$, and $\dot{\mathbbm 1}_\xi$, let $\beta, \eta$ be such that $f(\xi)=(\beta,\eta)$.
Notice that $\eta\leq\xi$.
Then $\dot{\preceq}^\eta_\beta$ and $\dot S^\eta_\beta$ are $\mathbb P_\eta$-names, and therefore $\mathbb P_\xi$-names.
By the bounded version of the maximal principle, choose
$\mathbb P_\xi$-names $\dot{\mathbb Q}_\xi$, $\dot{\leq}_\xi$, and
$\dot{\mathbbm 1}_\xi$ in $H(\lambda^+)$ such that
\[
  \mathbb P_\xi\Vdash
  (\dot{\mathbb Q}_\xi,\dot{\leq}_\xi,\dot{\mathbbm 1}_\xi)
  =
  Q_{(\omega,\dot{\preceq}^\eta_\beta)}
  \bigl(\Col(\dot S^\eta_\beta)\bigr).
\]
These names can be selected as the $\sqsubseteq$-least such names.
This completes the construction.

It follows by induction from Lemmas \ref{lem:tree-preserves-intertwined}
and \ref{lem:fs-iteration-preserves-intertwined} that, for every $\xi\leq\lambda$,
\[
  \mathbb P_\xi\Vdash
  \langle\check{\mathcal B},\check{\mathcal C}\rangle
  \text{ is intertwined}.
\]
In particular, $\mathbb P_\lambda\Vdash \langle\check{\mathcal B},\check{\mathcal C}\rangle$ is intertwined.
Thus, $\mathbb P_\lambda\Vdash \mathfrak{ap}=\omega_1$.

By Lemma~\ref{prop:at-le-b} and the inequality
$\mathfrak q\leq\mathfrak c$, it remains to prove that
$\mathbb P_\lambda\Vdash \mathfrak c\leq\check \lambda\leq\mathfrak{at}$.

We start by showing that $\mathbb P_\lambda\Vdash \mathfrak c\leq \check \lambda$.
It suffices to show that there are at most $\lambda$ many nice $\mathbb P_\lambda$-names for subsets of $\omega$.
As $\mathbb P_\lambda$ has the countable chain condition, as $\mathbb P_\lambda=\bigcup_{\xi<\lambda}\mathbb P_\xi$, and as each $\mathbb P_\xi$ has cardinality less than $\lambda$, we have $|\mathbb P_\lambda|\leq \lambda$.
Hence, since $\mathbb P_\lambda$ has the countable chain condition, there are at most
\[
  \bigl|[\mathbb P_\lambda]^\omega\bigr|^\omega
  \leq\lambda^\omega=\lambda
\]
nice $\mathbb P_\lambda$-names for subsets of $\omega$.
Thus, $\mathbb P_\lambda\Vdash \mathfrak c\leq \check \lambda$.

Now we show that $\mathbb P_\lambda\Vdash \check \lambda\leq \mathfrak{at}$.
It suffices to prove this holds relativized to a countable transitive model $M$.

Let $G$ be a $\mathbb P_\lambda$-generic filter over $M$.
Working in $M[G]$, it is enough to show that for every $\omega$-tree $T$ and every pair $\mathcal X,\mathcal Y$ of disjoint families of infinite subtrees of $T$ such that $\mathcal X\cup\mathcal Y$ is almost disjoint and $|\mathcal X\cup\mathcal Y|<\lambda$, there exists $D\subseteq T$ such that $D\cap X$ is finite for every $X\in\mathcal X$ and $D\cap Y$ is infinite for every $Y\in\mathcal Y$.
This is exactly the weak-separation statement for the almost disjoint family
$\mathcal X\cup\mathcal Y$ and its subcollection $\mathcal Y$.
If $\mathcal Y=\emptyset$, take $D=\emptyset$.
If $\mathcal X$ is finite, take $D=T\setminus\bigcup\mathcal X$.
Thus we may assume that $\mathcal X$ is infinite and that $\mathcal Y$ is nonempty.

We may also assume that the underlying set of $T$ is $\omega$, by transporting
the tree structure along a bijection.
Let $\preceq\subseteq \omega\times \omega$ be the order of $T$.

Let $\kappa_X=|\mathcal X|$ and $\kappa_Y=|\mathcal Y|$.
Let $S_X\subseteq \kappa_X\times \omega$ be such that $\Col(S_X)=\mathcal X$ and $S_Y\subseteq \kappa_Y\times \omega$ be such that $\Col(S_Y)=\mathcal Y$.
Fix $\mathbb P_\lambda$-names $\dot\preceq'$, $\dot S_X'$, and $\dot S_Y'$ such that
$\preceq=\dot\preceq'_G$, $S_X=(\dot S_X')_G$, and $S_Y=(\dot S_Y')_G$.
By the maximal principle, we may assume that $\mathbb P_\lambda\Vdash \dot S_X'\subseteq \check{\kappa_X}\times\check\omega$, that $\mathbb P_\lambda\Vdash \dot S_Y'\subseteq \check{\kappa_Y}\times\check\omega$, and that
\[
\begin{aligned}
  \mathbb P_\lambda\Vdash{}&(\omega,\dot\preceq')\text{ is an }\omega\text{-tree,}\\
  &\Col(\dot S_X')\text{ is an infinite almost disjoint family of subtrees of }(\omega,\dot\preceq'),\\
  &\text{and }\Col(\dot S_X')\cup\Col(\dot S_Y')\text{ is an almost disjoint family of infinite subtrees}\\
  &\text{of }(\omega,\dot\preceq').
\end{aligned}
\]
Let $\dot S_X$, $\dot S_Y$, and $\dot\preceq$ be, respectively, nice names for subsets of $\check{\kappa_X}\times\check\omega$, $\check{\kappa_Y}\times\check\omega$, and $\check{\omega\times\omega}$ such that $\mathbb P_\lambda\Vdash \dot S_X=\dot S_X'$, $\mathbb P_\lambda\Vdash \dot S_Y=\dot S_Y'$, and $\mathbb P_\lambda\Vdash \dot\preceq=\dot\preceq'$.

As $\mathbb P_\lambda$ is ccc, as $\lambda$ is regular, and as
$\mathbb P_\lambda=\bigcup_{\xi<\lambda}\mathbb P_\xi$, there exists
$\nu<\lambda$ such that $\dot{\preceq}$, $\dot S_X$, and
$\dot S_Y$ are $\mathbb P_\nu$-names.
Choose $\beta<\lambda$ such that
\[
  (\dot{\preceq}^\nu_\beta,\dot S^\nu_\beta,\kappa^\nu_\beta)
  =
  (\dot{\preceq},\dot S_X,\kappa_X)
\]
and choose $\xi<\lambda$ such that $f(\xi)=(\beta,\nu)$.
Then, by item~\ref{item:forcing-definitions},
\[
  \mathbb P_\xi\Vdash
  (\dot{\mathbb Q}_\xi,\dot{\leq}_\xi,\dot{\mathbbm 1}_\xi)
  =
  Q_{(\omega,\dot{\preceq})}
  \bigl(\Col(\dot S_X)\bigr).
\]

Thus, in $M[G_\xi]$, the forcing at coordinate $\xi$ is
$Q_T(\mathcal X)$.
Moreover, $\mathcal X,\mathcal Y\in M[G_\xi]$, and every
$Y\in\mathcal Y$ lies outside $\mathcal I_T(\mathcal X)$, so this ideal is
proper.  By Lemma~\ref{lem:QT-adds-separator}, the generic set added at stage
$\xi$ gives a set $D \in M[G_{\xi+1}]\subseteq M[G]$ such that $D\cap X$ is
finite for every $X\in\mathcal X$ and $D\cap Y$ is infinite for every
$Y\in\mathcal Y$, as intended.
\end{proof}

%% file: sections/conclusion.tex
\section{Concluding remarks}

In this paper, we settled Problem~12 of \cite{banakh2023q} by proving that
$\mathfrak{adp}=\mathfrak{dp}$.
We also obtained the analogous identity for the dual asymmetric version of this
cardinal, $\mathfrak{adp}_2=\mathfrak{ap}$.

Moreover, we answered Problem~11(1) of
\cite{banakh2023q} by proving $\mathfrak{ap}\leq\mathfrak{at}= \mathfrak q_{2}$.
In Section~\ref{sec:small-diagonals}, we further showed that
$\mathfrak q_{1\Delta}=\mathfrak{at}=\mathfrak q_2$ and that
$\mathfrak q_{2\frac{1}{2}\Delta}=\mathfrak q_{2\frac{1}{2}}$.

Finally, the forcing construction in Section~\ref{sec:ap-below-at-model}
shows that $\mathfrak{at}=\mathfrak q_{2}$ can consistently be strictly larger than
$\mathfrak{ap}$, a phenomenon similar to the one observed by Brendle \cite{brendle1999dow} for $\mathfrak{ap}$ and $\mathfrak q$.

The inequalities and equalities established or recalled among the cardinals considered in
this paper are summarized in Figure~\ref{fig:q-cardinals-diagram}.  An arrow
$\kappa\to\lambda$ means $\kappa\leq\lambda$.
The relevant references are listed in the table below.

\begin{center}
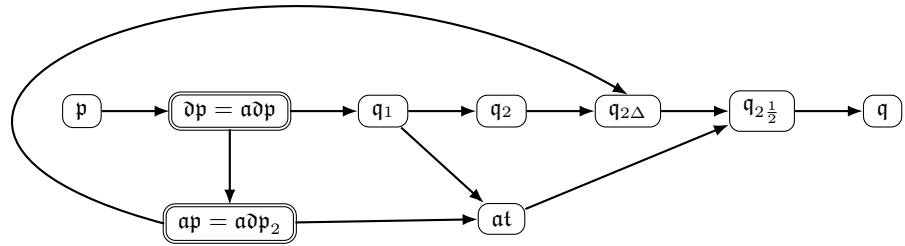

\resizebox{\textwidth}{!}{%
\begin{tikzpicture}[
  node distance=10mm and 9mm,
  cardinal/.style={
    draw,
    rounded corners,
    inner xsep=5pt,
    inner ysep=3pt,
    align=center,
    font=\small
  },
  eqcard/.style={
    cardinal,
    double,
    double distance=0.7pt
  },
  arrow/.style={-{Latex[length=2mm]}, thick}
]

\node[cardinal] (p) {$\mathfrak p$};
\node[eqcard, right=of p] (dp) {$\mathfrak{dp}=\mathfrak{adp}$};
\node[cardinal, right=of dp] (qone) {$\mathfrak q_1$};
\node[eqcard, right=of qone] (qtwo) {$\mathfrak q_{1\Delta}$\\$=\mathfrak{at}=\mathfrak q_2$};
\node[cardinal, right=of qtwo] (qtwodelta) {$\mathfrak q_{2\Delta}$};
\node[eqcard, right=of qtwodelta] (qtwohalf) {$\mathfrak q_{2\frac{1}{2}\Delta}$\\$=\mathfrak q_{2\frac{1}{2}}$};

\node[eqcard, below=of dp] (ap) {$\mathfrak{ap}=\mathfrak{adp}_2$};
\node[cardinal, right=of qtwohalf] (q) {$\mathfrak q$};

\draw[arrow] (p) -- (dp);
\draw[arrow] (dp) -- (qone);
\draw[arrow] (qone) -- (qtwo);
\draw[arrow] (qtwo) -- (qtwodelta);
\draw[arrow] (qtwodelta) -- (qtwohalf);
\draw[arrow] (qtwohalf) -- (q);

\draw[arrow] (dp) -- (ap);
\draw[arrow] (ap) -| (qtwo);

\end{tikzpicture}
}

\captionof{figure}{Nonstrict inequalities among the cardinals considered in this paper.}
\label{fig:q-cardinals-diagram}
\end{center}

\begin{center}
\begin{tabular}{@{}ll@{}}
\toprule
\textbf{Relation} & \textbf{Reference} \\
\midrule

$\mathfrak p\leq\mathfrak{dp}$ 
& Brendle \cite{brendle1999dow}. \\

$\mathfrak{dp}\leq\mathfrak{ap}$ 
& Brendle \cite{brendle1999dow}. \\

$\mathfrak{dp}=\mathfrak{adp}$ 
& Theorem~\ref{thm:adp-equals-dp}. \\

$\mathfrak{adp}_2=\mathfrak{ap}$ 
& Theorem~\ref{thm:adp2}. \\

$\mathfrak{adp}\leq\mathfrak q_1$ 
& Banakh--Bazylevych \cite{banakh2023q}. \\

$\mathfrak q_1\leq\mathfrak q_2$
& Banakh--Bazylevych \cite{banakh2023q}. \\

$\mathfrak q_2\leq\mathfrak q_{2\Delta}$ 
& By definition of $\mathfrak q_{2\Delta}$. \\

$\mathfrak q_{2\Delta}\leq\mathfrak q_{2\frac{1}{2}}$ 
& Lemma~\ref{lem:second-countable-urysohn-gdelta-diagonal}. \\

$\mathfrak q_{2\frac{1}{2}}\leq \mathfrak q$ 
& Banakh--Bazylevych \cite{banakh2023q}. \\

$\mathfrak{ap}\leq\mathfrak{at}$ 
& Definition~\ref{def:at}. \\

$\mathfrak{at}=\mathfrak q_2$
& Theorem~\ref{thm:at-equals-q2}. \\

$\mathfrak q_{1\Delta}=\mathfrak{at}$
& Theorem~\ref{thm:q1delta-equals-at}. \\

$\mathfrak q_{2\frac{1}{2}\Delta}=\mathfrak q_{2\frac{1}{2}}$
& Lemma~\ref{lem:second-countable-urysohn-gdelta-diagonal}. \\

$\mathfrak{at}\leq\mathfrak q\leq\mathfrak b$
& Lemma~\ref{prop:at-le-b} and \cite{miller1984special}. \\

\bottomrule
\end{tabular}
\end{center}

The following table records the known consistency results yielding strict
inequalities among the cardinals appearing in the diagram.

\begin{center}
\begin{tabular}{@{}ll@{}}
\toprule
\textbf{Consistent inequality} & \textbf{Reference} \\
\midrule

$\mathfrak p<\mathfrak{dp}$ 
& Dow, as quoted in Brendle \cite{brendle1999dow}. \\

$\mathfrak{dp}<\mathfrak{ap}$ 
& Brendle \cite[Theorem~A]{brendle1999dow}. \\

$\mathfrak{ap}<\mathfrak{at}=\mathfrak q_2$
& Theorems~\ref{thm:consistent-ap-below-at} and~\ref{thm:at-equals-q2}. \\

\bottomrule
\end{tabular}
\end{center}

We therefore record the following questions.

\begin{prob}\label{prob:remaining-q-cardinals}\hfill
    \begin{enumerate}[label=(\arabic*)]
        \item Is $\mathfrak q_1\leq \mathfrak{ap}$?
        \item Is $\mathfrak{ap}\leq \mathfrak q_1$?
        \item Is $q_{2\Delta}=\mathfrak q_{2\frac{1}{2}}$?
        \item Is $q_{2\Delta}=\mathfrak q_2$?
    \end{enumerate}
\end{prob}
\subsection*{Acknowledgements}
This study was financed by the São Paulo Research Foundation (FAPESP), Brazil.
Grant number:~\mbox{2025/07302-0}.

\subsection*{Declaration of generative AI and AI-assisted technologies}
During the preparation of this work, the author used OpenAI's GPT-6 Astra for grammar revision, proofreading, checking the mathematical arguments, brainstorming, generating first drafts of mathematical proofs and obtaining suggestions and insights. The author reviewed and edited the AI-assisted output as needed and takes full responsibility for the content of the published article.
During the preparation of the first version of this work, the author used OpenAI's GPT-5.5 for the same purposes.